\documentclass{article}

\usepackage{amsmath,amssymb,amsthm,color}
\usepackage[colorlinks=true]{hyperref}

\newcommand{\N}{\mathbb{N}}
\newcommand{\R}{\mathbb{R}}
\renewcommand{\d}{\mathrm{d}}

\newcommand{\D}{\mathcal{D}}
\newcommand{\h}{\mathcal{H}}
\newcommand{\s}{\mathcal{S}}
\renewcommand{\leq}{\leqslant}
\renewcommand{\geq}{\geqslant}

\newtheorem{theorem}{Theorem}  
\newtheorem{proposition}{Proposition}

\newtheorem{definition}{Definition}
\newtheorem{lemma}{Lemma}
\theoremstyle{definition}\newtheorem{example}{Example}
\theoremstyle{definition}\newtheorem{remark}{Remark}

\title{The general setting for Shape Analysis}

\date{}

\author{Sylvain Arguill\`ere\footnote{Center for imaging science, Johns Hopkins University, Baltimore, MD, USA ({\tt sarguil1@johnshopkins.edu}).}
}

\begin{document}

\maketitle

\begin{abstract}
In shape analysis, the concept of shape spaces has always been vague, requiring a case-by-case approach for every new type of shape. In this paper, we give a general definition for an abstract space of shapes in a manifold $M$. This notion encompasses every shape space studied so far in the literature, and offers a rigorous framework for several possible generalizations. We then give the appropriate setting for LDDMM methods of shape analysis, which arises naturally as a sub-Riemannian structure on a shape space. This structure is deduced from the space of infinitesimal deformations and their infinitesimal action. We then describe the properties of the Hamiltonian geodesic flow, and study several applications of equivariant mappings between shape spaces.
\end{abstract}


\section*{Introduction}


Mathematical shape analysis is a relatively recent area of study, whose goal is to compare several shapes while keeping track of their geometric properties. This is done by finding a deforming an initial shape (a \emph{template}) into another target shape in a way that minimizes a certain cost that depends on the properties of the shapes. This implies that a cost has been assigned to every possible deformation of the template, the design of this cost function being a crucial step in the method. This technique has lead to great contributions in computational anatomy (see \cite{GM}).

A powerful method, the Large Deformation Diffeomorphic Metric Mapping (LDDMM), represents deformations as flows of diffeomorphisms on the ambient space $M$ in which the shape is embedded, this flow being generated by certain time-dependent vector fields \cite{DGM,T1,T2} on the ambient space. These vector fields are chosen at each time in a fixed space $V$ of "infinitesimal transformations", equipped with a Hilbert product $\langle\cdot,\cdot\rangle$. In other words, we have a deformation $t\mapsto q(t)$ of the initial shape $q(0)=q_0$ given by
$$
q(t)=\varphi(t)\cdot q_0,
$$
where $\varphi(t)$ is the flow at time $t$ of the time-dependent vector field $t\mapsto X(t)\in V$: $\partial_t\varphi(t,x)=X(t,\varphi(t,x))$. The cost of this deformation is then defined by the energy, or action, of the deformation $\varphi(\cdot)$, given by
$$
\int \langle X(t),X(t)\rangle \d t.
$$
The goal is then to get the initial shape close to the target shape while minimizing this cost. Note that these transformations preserve local (such as the smoothness) and global (such as the number of self-intersections) geometric properties of the deformed shape.


One of the main issues with mathematical shape analysis is that the concept of a ``shape space" is rather vague: the literature only studies \emph{examples} of shape spaces (usually spaces of embeddings of a manifold $N$ in $M$, or just spaces of submanifolds in $M$ \cite{BHM,TY2}), without giving a good definition of what a shape space is. This not only forces a case by case approach, it can also limit the kind of shapes that are studied. For example, none of the existing methods tackle the issue of fibered shapes (submanifolds of $M$ coupled with a fixed tangent vector field, representing the direction of a fiber), which would be very useful when studying the movement of a muscle, for example. This paper aims to remedy this situation by giving an abstract definition of a shape space that unifies and generalizes all the different shape spaces studied so far, and to generalize to this setting all the tools of the LDDMM framework. 

In order to define the space of infinitesimal transformations and the corresponding norm, two ways are classically considered in the existing literature. The first, used for example in \cite{BHM,BBHM,MM}, consists of considering all smooth vector fields with compact support. The resulting structure is a weak right-invariant Riemannian structure on the Lie group of smooth diffeomorphisms with compact support on $M$ \cite{KMBOOK}, which projects into a weak Riemannian metric on the shape space. 

The second, which is the point of view adopted throughout most of this paper, consists of considering an (a priori arbitrary) Hilbert space $(V,\langle\cdot,\cdot\rangle)$ of vector fields on $M$ with continuous inclusion in the space of vector fields vector fiels of high enough Sobolev regularity. The space $V$ is usually defined using a corresponding \emph{reproducing kernel} (see Section 1.2). This is the framework that lead to the development of the so-called LDDMM methods (see \cite{AG,AGH,AGR,BMTY,CCT,JM,MTY1,MTY2,YBOOK}). The shape spaces are then usually studied as \emph{strong Riemannian manifolds}, for which deformations with minimal costs are, again, geodesics. 

However, we will see that in this context, the shape spaces were actally \emph{sub-Riemannian}, and not Riemannian. In the existing literature, the authors were not familiar with sub-Riemannian geometry, which is why the subject was treated from the Riemannian point of view\footnote{This does not make their result incorrect though: first of all, in these methods, geodesics are usually found using the Hamiltonian viewpoint. Second of all, in practical applications, the shape spaces are often finite dimensional spaces of landmarks, for which the structure is actually Riemannian.}. 
Moreover, since the shape spaces we are dealing with are infinite dimensional, we get infinite-dimensional sub-Riemannian manifolds. This complicates matters even more, but can be worked around, as we will see.

The paper is organized as follows. In Section 1, we recall the results of \cite{AT} on right-invariant sub-Riemannian structures on the group $\mathrm{Diff}(M)$ of diffeomorphisms of a smooth manifold $M$ with bounded geometry, which will be necessary for all the results of this paper. Then, in Section 2, we give the definition of an abstract shape space in $M$, deduce the sub-Riemannian structure it inherits from the action of $\mathrm{Diff}(M)$, and study the corresponding notions of distance, geodesics, singular curves, and their applications to LDDMM methods. We also give some examples of Hamiltonian geodesic equations for several shape spaces, some that were already known but treated from a Riemannian viewpoint, some completely new. Finally, in Section 3, we look for symmetries in the geodesic flow, and study various applications of equivariant mappings between shape spaces. We also define the notion of \emph{lifted shape spaces} which generalize the concept of diffeons from \cite{Y1}, and give a few examples of geodesic equations on finite dimensional lifted shape spaces.

\section{Sub-Riemannian structures on groups of diffeomorphisms}\label{sec2}
In shape analysis, and in the LDDMM methods in particular, one uses flows of vector fields on the ambient space $M$ to compare shapes. Moreover, these vector fields belong to a fixed, arbitrary Hilbert space $V$. However, as shown in \cite{AT}, such a setting naturally endowes the group of diffeomorphisms of $M$ with a \emph{sub-Riemannian structure}. The purpose of this section is to describe this structure and to give a brief summary of the results of \cite{AT}. 

\subsection{The group of diffeomorphisms}
We consider $(M,g^M)$ and $(N,g^N)$ smooth oriented Riemannian manifolds with bounded geometry. Recall that this means that their global injectivity radius is positive and that for every $i\in\N$, the Riemannian norm of the $i$-th covariant derivative of their curvature tensor is bounded. Under these assumptions, for any integer  $s>\dim(M)/2$, one can define a smooth Hilbert structure on the space $H^s(M,N)$ of maps of Sobolev class $H^s$ from $M$ to $N$ (see \cite{AT,EM,S}).  These spaces are metrisable and we denote $d_{H^s(M,N)}$ the corresponding distances.

\begin{remark}
We obtain a decreasing sequence of Hilbert manifolds with dense continuous inclusions $H^{s+1}(M,N)\hookrightarrow H^s(M,N)$, so that the limit $H^\infty(M,N)$ is a smooth Inverse Limit Hilbert (ILH) manifold \cite{O}.
\end{remark}

For $s>d/2+1$, let $\mathcal{D}^s(M)=(H^s(M,M)\cap\mathrm{Diff}(M))_e$ be the connected component of $e=\mathrm{Id}_M$ in the space of $\mathcal{C}^1$-diffeomorphisms of $M$ that also belong to $H^s(M,M)$. Then $\mathcal{D}^s(M)$ is an open subset of $H^s(M,M)$, and hence a Hilbert manifold. Its tangent space at $e$ is given by $T_e\D^s(M)=\Gamma^s(TM)$ the space of all vector fields on $M$ of Sobolev class $H^s$. 
It is given everywhere else by $T_\varphi\D^s(M)=\Gamma^s(TM)\circ\varphi$. It is also a group for the composition $(\varphi,\psi)\mapsto\varphi\circ\psi$. This group law satisfies the following conditions: 
\begin{enumerate}
\item \textbf{Continuity:} $(\varphi,\psi)\mapsto \varphi\circ\psi$ is continuous.
\item \textbf{Smoothness on the left:} For every $\psi\in \D^s(M)$, the mapping $R_\psi:\varphi\mapsto\varphi\circ\psi$ is smooth.
\item \textbf{Smoothness on the right:} For every $k\in \N\setminus\{0\}$, the mappings
\begin{equation}\label{def_xi}
\begin{array}{rclcrcl}
\mathcal{D}^{s+k}(M)\times\D^s(M)&\longrightarrow& \D^s(M)&\quad & \Gamma^{s+k}(TM)\times\D^s(M) &\longrightarrow &T\D^s(M) \\
(\varphi,\psi)&\longmapsto&\varphi\circ \psi&\quad &(X,\psi)&\longmapsto&X\circ\psi
\end{array}
\end{equation}
are of class $\mathcal{C}^{k}$.
\end{enumerate}
\begin{remark} In particular, the intersection $\D^\infty(M)$ of the sequence $(\mathcal{D}^s(M))_{s>d/2+1}$ is equipped with an ILH Lie group structure, on which composition of diffeomorphisms is smooth (and the above conditions are still valid for $s=\infty$) \cite{ES,O,S}. 
\end{remark}

\paragraph{Curves on $\mathcal{D}^s(M)$.} Let $s$ be an integer greater than $d/2+1$. 
For every $\varphi(\cdot)\in H^1(0,1;\mathcal{D}^s(M))$, define the \emph{logarithmic velocity} of $\varphi(\cdot)$ by 
$$
X(\cdot) = {\dot\varphi(\cdot)}\circ\varphi(\cdot)^{-1} \in L^2(0,1;\Gamma^s(TM)).
$$
Note that, by construction, we have $\dot\varphi(t) = X(t)\circ\varphi(t)$ for almost every $t\in[0,1]$.

In other words, any curve $\varphi(\cdot)\in H^1(0,1;\mathcal{D}^s(M))$ of diffeomorphisms is the flow of a time-dependent right-invariant vector field on $\mathcal{D}^s(M)$ with square-integrable $H^s$-norm. Conversely, thanks to the measurable version of the Cauchy-Lipschitz theorem (see, e.g.,  \cite{TBOOK}), any time-dependent vector field $X(\cdot)\in L^2(0,1;\Gamma^s(TM))$ generates a unique flow $\varphi^X(\cdot)\in H^1(0,1;\mathcal{D}^s(M))$ such that $\varphi^X(0)=e$, and therefore a unique curve $\varphi(\cdot)=\varphi^X(\cdot)\circ\varphi_0$ for any fixed intitial condition $\varphi_0\in\mathcal{D}^s(M)$. A group that satisfies such properties is called \textbf{regular} \cite{KMBOOK}. In other words, $\D^s(M)$ satisfies the additional property:
\vspace{2mm}

{4. \textbf{Regularity:} For any $\varphi_0\in\mathcal{D}^s(M)$ and $X(\cdot)\in L^2(\Gamma^s(TM)$, there is a unique curve $\varphi(\cdot)\in$

\hspace{4.2mm}$H^1(0,1;\D^s(M))$ such that $\varphi(0)=\varphi_0$ and $\dot{\varphi}(t)=X(t)\circ\varphi(t)$ almost everywhere on $[0,1]$.}

\subsection{Sub-Riemannian structures on $\mathcal{D}^s(M)$.}\label{sec_SRstructure}
Let us first define general sub-Riemannian structures on a Banach manifold.
\begin{definition}[\cite{A1}]
Let $\mathcal{M}$ be a smooth Hilbert manifold. A strong sub-Riemannian structure of class $\mathcal{C}^k$ on $\mathcal{M}$, $k\in \N\cup\{\infty\}$, is a triple $(\mathcal{E},g,\xi)$, where $\pi:\mathcal{E}\rightarrow M$ is a smooth Hilbert vector bundle on $M$ endowed with a smooth, strong  (i.e. fiberwise Hilbert) Riemannian metric $g$, and $\xi:\mathcal{E}\rightarrow TM$ a vector bundle morphism of class $\mathcal{C}^k$.

A \textbf{horizontal system} is a curve $t\in[0,1]\mapsto (q(t),u(t))\in \mathcal{E}$ of class $L^2$, such that its projection $t\mapsto q(t)$ to $\mathcal{M}$ is of Sobolev class $H^1$ and satisfies for almost every $t\in [0,1]$, $\dot{q}(t)=\xi_{q(t)}u(t)$. Its \textbf{length} and \textbf{action} are respectively defined by
$$
L(q,u)=\int_0^1\sqrt{g_{q(t)}(u(t),u(t))}dt,\quad A(q,u)=\frac{1}{2}\int_0^1g_{q(t)}(u(t),u(t))dt.
$$
A \textbf{horizontal curve} is the projection $q(\cdot)$ to $\mathcal{M}$ of a horizontal system. 
\end{definition}
Note that, if $\xi$ is one-to-one, there is a bijective correspondance between horizontal systems and horizontal curves.

\begin{definition} 
For $\mathcal{M}=\mathcal{D}^s(M)$, we can define strong right-invariant structures as follows: fix $V$ an arbitrary Hilbert space of vector fields with Hilbert product $\langle\cdot,\cdot\rangle$ and continuous inclusion in $\Gamma^{s+k}(TM),\ k\in \N$. The \textbf{sub-Riemannian structure} induced by $V$ on $\D^s(M)$ is 
$$
\mathrm{SR}_V(s)=(\D^s(M)\times V,\langle\cdot,\cdot\rangle,dR),
$$
where, for $(\varphi,X)\in\D^s(M)\times V$,
$$
dR_\varphi(e). X=dR_\varphi X=X\circ\varphi\ \in T_\varphi D^s(M).
$$
$\mathrm{SR}_V(s)$ is a strong sub-Riemannian structure of class $\mathcal{C}^k$. See \cite{AT} for various examples.
\end{definition}

\begin{remark}\label{rkhs}
Let $V^*$ denote the dual space of $V$, and $K_V:V^*\rightarrow V$ the inverse of the canonical isometry $X\in V\mapsto\langle X,\cdot\rangle\in V^*$. If $V$ has continuous inclusion in the space of continuous vector fields on $M$, then $K_V$ is given by convolution with a \emph{reproducing kernel} $K$. Such a kernel is a mapping $(x,y)\in M^2\mapsto K(x,y),$ with $K(x,y)$ an linear morphism from $T^*_yM$ to $T_xM$. It is defined by the property that for every $(y,p)\in T^*M,$ the vector field $x\mapsto K(x,y)p$ belongs to $V$ and
$$
\forall X\in V,\quad \langle K(\cdot,x)p,X\rangle=p(X(x)).
$$
Then, if an element $P$ of the dual $V^*$ can be represented by a one-form with (distributional) coefficients\footnote{For example, $\Gamma^{s}(TM)^*=\Gamma^{-s}(T^*M)$ can be identified with the space of one-forms on $M$ with coefficients in $H^{-s}(M)$. But since we assume that $V$ has continuous inclusion in $\Gamma^{s}(TM)$, any such 1-form also belongs to $V^*$.}  so that we can write
$$
P(X)=\int_M P(x)(X(x))dx,
$$
then the vector field $Y\in V$ given by
$$
Y(x)=\int_M K(x,y)P(y)dy
$$
satisfies
$$
\forall X\in V,\quad \langle Y,X\rangle=P(X).
$$
Conversely, such a mapping $K$ is the reproducing kernel of a unique Hilbert space of vector fields with continuous inclusion in the space of continuous vector fields if and only if it satisfies
$$
\iint_{M\times M} d\nu(x)(K(x,y)d\nu(y))=0\ \Rightarrow \forall x\in M,\ \int_{M}K(x,y)d\nu(y)=0
$$
for every compactly supported 1-form  $\nu$ with coefficients in the space of Radon measures on $M$ \cite{YBOOK}.
\end{remark}

\paragraph{Horizontal curves and end-point mapping.} The flow $\varphi^X(\cdot)$ of any $X(\cdot)\in L^2(0,1;V)$ is a horizontal curve for the structure defined in the previous paragraph, and conversely any horizontal curve $\varphi(\cdot)\in H^1(0,1;\D^s(M))$ with $\varphi(0)=e$ is  the flow of $X(\cdot)=\dot{\varphi}(\cdot)\circ\varphi^{-1}(\cdot)\in L^2(0,1;V)$. Hence, we can identify the set of all such horizontal curves with the Hilbert space $L^2(0,1;V)$ through the correspondance $X\leftrightarrow\varphi^X$. Note that the length and action of such a curve are given respectively by
$$
\begin{aligned}
L(\varphi^X(\cdot))&=L(X(\cdot))=\int_0^1\sqrt{\langle X(t),X(t)\rangle} dt,\\ A(\varphi^X(\cdot))&=A(X(\cdot))=\frac{1}{2}\int_0^1{\langle X(t),X(t)\rangle} dt.
\end{aligned}
$$

The endpoint mapping $\mathrm{end}:L^2(0,1;V)\rightarrow\D^s(M)$ is given by $\mathrm{end}(X(\cdot))=\varphi^X(1)$ and is of class $\mathcal{C}^k$ (recall that $k$ is such that $V\hookrightarrow \Gamma^{s+k}(TM)$ is a continuous inclusion). 


\paragraph{Sub-Riemannian distance.}
The \emph{sub-Riemannian distance} $d_{SR}$ between two elements $\varphi_0$ and $\varphi_1$ of $\mathcal{D}^s(M)$ is defined as the infimum of the length of horizontal curves steering $\varphi_0$ to $\varphi_1$, with the agreement that $d_{SR}(\varphi_0,\varphi_1)=+\infty$ whenever there is no horizontal curve steering $\varphi_0$ to $\varphi_1$. A curve $\varphi(\cdot):[0,1]\rightarrow \mathcal{D}^s(M)$ is said to be \emph{minimizing} if $d_{SR}(\varphi(0),\varphi(1)) = L(\varphi(\cdot))$. As usual, a curve minimizing the energy among all other curves with the same endpoints also minimizes the length.

\begin{proposition}\label{propo:finertopo} The sub-Riemannian distance $d_{SR}$ is indeed a distance (taking its values in $[0,+\infty]$), that is, $d_{SR}(\varphi_0,\varphi_1)=0$ implies $\varphi_0=\varphi_1$. Moreover, the corresponding topology on $\D^s(M)$ if at least as fine as the intrinsic manifold topology.
\end{proposition}

We also have the following generalization of Trouv\'e's result \cite{T1}.

\begin{theorem}[Completeness of the distance\cite{AT}]\label{th:complete}
Any two elements $\varphi_0$ and $\varphi_1$ of $\mathcal{D}^s(M)$ with $d_{SR}(\varphi_0,\varphi_1)<+\infty$ can be connected by a minimizing horizontal curve, and $(\mathcal{D}^s(M),d_{SR})$ is a complete metric space.
\end{theorem}

\subsection{Geodesic flow on $\mathcal{D}^s(M)$}\label{sec_geod}

We keep the framework and notations used in the previous sections, but we assume $k\geq1$. A \textit{geodesic} $\varphi^X(\cdot)$ starting from $e$ is a critical point of the action $A(X(\cdot))$ among all curves $\varphi^Y(\cdot)$ with the same endpoints. In other words, for every $\mathcal{C}^1$ variation $a\in(-\varepsilon,\varepsilon)\mapsto X_a(\cdot)\in L^2(0,1;V)$ such that $\mathrm{end}(X_a(\cdot))=\varphi_1$, and $X_0=X$, we have
$$
\partial_a(A(X_a(\cdot))_{\vert a=0}=0.
$$

\paragraph{The three kinds of geodesics.} 
It easy to see that for any such curve, the couple of linear operators 
$$
(\d A(X(\cdot)),\d\,\mathrm{end}(X(\cdot))):L^2(0,1;V)\rightarrow\R\times T_{\varphi_1}\D^s(M)
$$
is \textbf{not} onto. This can imply one of two possibilities:
\begin{enumerate}
\item There exists $(\lambda,P_{1})\in \R\times T^*_{\varphi_1}\mathcal{D}^s(M)\setminus\{(0,0)\}$ such that 
\begin{equation}\label{Lag_mult}
\d\,\mathrm{end}(X(\cdot))^*P_{1} + \lambda \d A(X(\cdot)) =0.
\end{equation}
This is a Lagrange multipliers relation, and can be used to give a Hamiltonian caracterization of $\varphi^X(\cdot)$. We have two subcases:
\begin{enumerate}
\item \emph{Normal case:} $\lambda\neq 0$. In that case, we can take $\lambda=-1$, so that $\d A(X(\cdot)) = \d\,\mathrm{end}(X(\cdot))^*P_{1}$. We say that $\varphi^X(\cdot)$ is a \textit{normal geodesic}.

Note that, conversely, for any $X(\cdot)\in L^2(0,1;V),$ if we have $\d A(X(\cdot)) = \d\,\mathrm{end}(X(\cdot))^*P_{1}$ for some $P_1$, then $\varphi^X(\cdot)$ is indeed clearly a geodesic.
\item \emph{Abnormal case:} $\lambda=0$. In that case, \eqref{Lag_mult} implies that $\varphi^X(\cdot)$ is a singular curve. We say that $\varphi^X(\cdot)$ is an \textit{abnormal geodesic}.
\end{enumerate}

\item The image of $(\d A(X(\cdot)),\d\,\mathrm{end}(X(\cdot)))$ is a proper dense subset of $T_{\varphi_1}\D^s(M)$. There are no Lagrange multipliers. We say that $\varphi^X(\cdot)$ is an \textit{elusive geodesic}. No Hamiltonian caracterization can be given.
\end{enumerate}

Note that this last case only happens because we work in infinite dimensions, and is a well-known problem in the community of control theory \cite{HeintzeLiu_AM1999,Kurcyusz,LiYong}. It is caused by a topological incompatibility between the manifold topology of $\D^s(M)$ and that of the sub-Riemannian distance $d_{SR}$. See \cite{A1,AT} for a more comprehensive discussion on the appearance of elusive geodesics and their consequences. Also see \cite{ChitourJeanTrelatJDG,ChitourJeanTrelatSICON,Bellaiche,MRBOOK,RT} for more on abnormal geodesics.

\begin{remark}
When studying shape spaces, $V$ is usually a dense subset of $\Gamma^s(TM)$. This implies that $\d\,\mathrm{end}(X(\cdot))$ has at least dense for all $X(\cdot)$. In that case, there are no abnormal geodesics, only normal and elusive ones.
\end{remark}

\paragraph{Hamiltonian formulation.} For $\lambda\in\{0,1\}$, define the Hamiltonian of the motion $H_\lambda(\varphi,P,X)=P(X\circ\varphi)-\frac{\lambda}{2}\langle X,X\rangle$, with $(\varphi,P)\in T^*\D^s(M)$ and $X\in V$. Then one can prove the following \cite{AT} partial Pontryagin maximum principle \cite{P}.
\begin{proposition}\label{prop:hamiltcardiffeo}
Let $\varphi(\cdot)$ be the flow of $X(\cdot)\in L^2(0,1;V)$, with $\varphi(1)=\varphi_1$, and $(\lambda,P_1)\in\{0,1\}\times T^*_{\varphi_1}\D^s(M)\setminus\{(0,0)\}$. Let $(\varphi(\cdot),P(\cdot))$ be the unique lift of $\varphi(\cdot)$ to $T^*\D^s(M)$ that solves the linear Cauchy problem $P(1)=P_1$ and $\dot{P}(t)=-\partial_\varphi H_\lambda(\varphi(t),P(t),X(t))$ for almost every $t\in[0,1]$. Then 
\begin{equation}\label{hamil}
\lambda \d A(\varphi(\cdot)) = \d\,\mathrm{end}(X(\cdot))^*P_{1}\quad \Longleftrightarrow\quad K_VdR_{\varphi(t)}^*P(t)=\lambda X(t)\ \ \text{a.e.}\ t\in[0,1].
\end{equation}
In this case, we call $P(\cdot)$ a \textbf{singular covector} when $\lambda=0$ and a \textbf{normal covector} when $\lambda=1$.
\end{proposition}
In particular, a curve $\varphi(\cdot)$, flow of $X(\cdot)$, is singular if and only if it can be lifted to a curve $(\varphi(\cdot),P(\cdot))$ such that, for almost every $t\in[0,1]$,
$$
\dot{P}(t)=-(dX(t)\circ\varphi(t))^*P(t),\quad V\subset\ker\ P(t).
$$

On the other hand, define the normal Hamiltonian $H:T^*\D^s(M)\rightarrow\R$ by $$H(\varphi,P)=H_1(\varphi,P,K_VdR_\varphi^*P).$$ 
$H$ is of class at least $\mathcal{C}^k$. We see that normal geodesics are the projections to $\D^s(M)$ of solutions on $T^*\D^s(M)$ of the so-called \textit{Hamiltonian geodesic equation} $(\dot{\varphi},\dot{P})=(\partial_PH(\varphi,P),-\partial_\varphi H(\varphi,P))$. Note that, using the reproducing kernel $K$ of $V$, we have
$$
H(\varphi,P)=\frac{1}{2}\int_{M\times M}P(x)(K(\varphi(x),\varphi(y))P(y))\d x\d y ,
$$
so that we get the usual formulas
$$\begin{aligned}
\partial_PH(\varphi,P)(x)&=\int_{M}K(\varphi(x),\varphi(y))P(y)\d x\d y,\\
\partial_\varphi H(\varphi,P)(x)&=\frac{1}{2}\int_M\left(\partial_1K(\varphi(x),\varphi(y))P(y)\right)^*P(x)\d y+\frac{1}{2}\int_M(\partial_2K(\varphi(y),\varphi(x))P(x))^*P(y)\d y.
\end{aligned}
$$
Moreover, one easily checks that $(\varphi,P)\rightarrow(\partial_PH(\varphi,P),-\partial_\varphi H(\varphi,P))$ is a vector field of class $\mathcal{C}^{k-1}$ on $T^*\D^s(M)$. We obtain the following theorem.
\begin{theorem}
If $k\geq2$, then the Hamiltonian geodesic flow is well-defined, of class $\mathcal{C}^{k-1}$, and global. In other words, for every $(\varphi_0,P_0)\in T^*\D^s(M),$ there is a unique solution $(\varphi(\cdot),P(\cdot)):\R\rightarrow T^*\D^s(M)$ to the Cauchy problem
$$
(\varphi(0),P(0))=(\varphi_0,P_0),\quad (\dot{\varphi}(t),\dot{P}(t))=(\partial_PH(\varphi(t),P(t)),-\partial_\varphi H(\varphi(t),P(t)))\ \ \text{a.e.}\ t\in[0,1].
$$
This solution is of class $\mathcal{C}^k$ in time and $\mathcal{C}^{k-1}$ in the initial conditions $(\varphi_0,P_0)$. Moreover, any subarc of this solution projects, up to re-parametrization, to a normal geodesic on $\D^s(M)$. Conversely, any normal geodesic is the projection of such a solution.
\end{theorem}

\paragraph{Momentum formulation.}\label{momentum}
We define the \emph{momentum map} $\mu:T^*\mathcal{D}^s(M)\rightarrow\Gamma^{s}(TM)^*=\Gamma^{-s}(T^*M)$ by $\mu(\varphi,P)=(\d R_\varphi)^*.P$. This proposition is proven in \cite{AT}, and is connected to the EPDiff equation \cite{A,EM,HMR,MM}.

\begin{proposition}\label{prop_momentum}
Assume that $k\geq1$. Then $\varphi(\cdot)\in H^1(0,1;\mathcal{D}^s(M))$ be either a normal geodesic or a singular curve, with logarithmic velocity $X(\cdot)\in L^2(0,1;V)$, and let $P(\cdot)$ be a corresponding normal or singular covector. We denote by $\mu(t)=\mu(\varphi(t),P(t))$ the corresponding momentum along the trajectory. Then, for almost every $t\in[0,1]$, we have
$$
\dot{\mu}(t)=\mathrm{ad}^*_{X(t)}\mu(t)=-\mathcal{L}_{X(t)}\mu(t).
$$
Here, $\mathrm{ad}_X:\Gamma^{s+1}(TM)\rightarrow\Gamma^{s}(TM)$, with $\mathrm{ad}_XY=[X,Y]$, and $\mathcal{L}_X$ the Lie derivative with respect to $X$. As a consequence,
 $$
\mu(t)=\varphi(t)_*\mu(0),
$$
for every $t\in[0,1]$.
\end{proposition}

In this paper, we are also interested in the converse.

\begin{proposition}\label{prop_momentum2}
We assume that $k\geq1$. Fix $\varphi(\cdot)\in H^1(0,1;\mathcal{D}^s(M))$ a horizontal curve with logarithmic velocity $X(\cdot)\in L^2(0,1;V)$. Let $\lambda\in\{0,1\}$ and $t\in[0,1]\mapsto \mu(t)\in \Gamma^s(TM)^*$ be continuous such that for almost every $t\in[0,1]$, $K_V\mu(t)=\lambda X(t)$ and $\dot{\mu}(t)=\mathrm{ad}^*_{X(t)}\mu(t)$.
Then $\varphi(\cdot)$ is a normal geodesic if $\lambda=1$ and a singular curve if $\lambda=0$, with covector $P(t)=(dR_{\varphi(t)}^*)^{-1}\mu(t).$
\end{proposition}
\begin{proof}
We already have $K_VdR_{\varphi(t)}^*P(t)=\lambda X(t)$ for almost every $t$. We just need to prove that $\dot{P}(t)=-\partial_\varphi H^\lambda(\varphi(t),P(t),X(t))=-(dX(t)\circ\varphi(t))^*P(t)$. Let $Y\in \Gamma^{s+1}(TM)$. We have
$$
\begin{aligned}
\dot{P}(t)(Y\circ\varphi(t))&=\frac{d}{dt}(\underset{=\mu(t)(Y)}{\underbrace{{P}(t)(Y\circ\varphi(t)))}}-{P}(t)(dY\circ\varphi(t).X(t)\circ\varphi(t))\\
&=\mu(t)([X(t),Y])-\mu(t)(dY.X(t))=-\mu(t)(dX(t).Y)\\
&=-P(t)(dX(t)\circ\varphi(t).Y\circ\varphi(t))\\
&=-(dX(t)\circ\varphi(t))^*P(t)(Y\circ\varphi(t)).
\end{aligned}
$$
Consequently, $\dot{P}(t)=-(dX(t)\circ\varphi(t))^*P(t)$ on $\Gamma^{s+1}(TM)\circ\varphi(t)$, and therefore on all of $T_\varphi \D^s(M)=\Gamma^{s}(TM)\circ\varphi(t)$ by density.
\end{proof}

\section{Shape spaces}\label{shsp}

In this section, we give the definition for abstract shape spaces in $M$ which allow the application of the usual LDDMM methods. We want the basic examples of shape spaces, such as spaces of landmarks, or spaces of embeddings of a fixed submanifold $N$, to be included in this definition. A common factor between those examples is that they are acted on by $\D^s(M)$, and this action satisfies properties that are very similar to the particular topological group structure of $\D^s(M)$. This action, along with the choice of a Hilbert space $V$ of vector fields of $M$, endowes the examples with a length structure (considered Riemannian structure in the literature, but which is often only a \emph{sub-Riemannian structure}), and the problem of comparing shapes can then be seen as a search for geodesics with respect to this structure. This leads to a Hamiltonian equation for the geodesic flow,  which in turn lead to various optimization methods \cite{TY2,YBOOK}.

Therefore, we define shape spaces as Banach manifolds on which the group of diffeomorphisms of $M$ acts in a way that is ``compatible" with this particular group structure. This then lets us project the sub-Riemannian structure induced on $\D^s(M)$ by an arbitrary Hilbert space of vector fields $V$ to a sub-Riemannian structure on the shape space itself, for which horizontal curves are exactly curves induced by flows of elements of $V$. We study the corresponding sub-Riemannian distance, and prove in particular the geodesic completeness of the structure. Then, we find the Hamiltonian geodesic flow and the singular curves on the shape space for this structure, and their connection with the Hamiltonian geodesic flow and the singular curves of $\D^s(M)$ itself. This allows the generalization  of the usual LDDMM methods and algorithms  to this broader category of shape spaces.

\subsection{Definition}
Throughout the section, $M$ is a smooth Riemannian manifold of dimension $d$ and bounded geometry. Let $s_0$ be the smallest integer such that $s_0>d/2$. A shape space in $M$ is a Banach manifold acted upon by $\D^s(M)$ for some $s$ in a way that is compatible with its particular topological group structure. 

\begin{definition}\label{shapespace}
Let  $\mathcal{S}$ be a Banach manifold and $\ell\in\N\setminus\{0\}$, and $s=s_0+\ell$. Assume that $\mathcal{D}^{s}(M)$ acts on $\s$, according to the action
\begin{equation}
\label{eq:groupaction}
\begin{array}{rcl}
\mathcal{D}^s(M)\times \mathcal{S}&\rightarrow& \mathcal{S}\\
(\varphi,q) &\mapsto& \varphi\cdot q
\end{array}
\end{equation}
We say that $\mathcal{S}$ is a shape space of order $\ell$ in $M$ if the following conditions are satisfied:
\begin{enumerate}
\item \textbf{Continuity:} $(\varphi,q)\mapsto\varphi\cdot q$ is continuous.
\item \textbf{Smoothness on the left:} For every $q\in \mathcal{S}$, the mapping $R_q:\varphi\mapsto\varphi\cdot q$ is smooth. Its differential at $e=\mathrm{Id}_{\R^d}$ is denoted $\xi_q$, and is called the \textbf{infinitesimal action} of $\Gamma^s(TM)$.
\item \textbf{Smoothness on the right:} For every $k\in\N$, the mappings
\begin{equation}
\begin{array}{rclcrcl}
\mathcal{D}^{s+k}(M)\times\mathcal{S}&\longrightarrow& \mathcal{S}&\quad &\xi : \Gamma^{s+k}(TM)\times\mathcal{S} &\longrightarrow &T\mathcal{S} \\
(\varphi,q)&\longmapsto&\varphi\cdot q&\quad &(X,q)&\longmapsto&\xi_qX
\end{array}
\end{equation}
are of class $\mathcal{C}^{k}$.
\item \textbf{Regularity:} For every $X(\cdot)\in L^2(0,1;\Gamma^s(TM))$ and $q_0\in \s$, there exists a unique curve $q(\cdot)=q^X(\cdot)\in H^1(0,1;\s)$ such that $q^X(0)=q_0$ and $\dot{q}^X(t)=\xi_{q^X(t)}X(t)$ for almost every $t$ in $[0,1]$.
\end{enumerate}
A an element $q$ of $\s$ is called a \textbf{state} of the shape.
\end{definition}
\begin{remark}
One can define shape spaces of order $\ell=0$ (see \cite{ATY}), but the action must still be performed by $\D^{s_0+1}(M)$.
\end{remark}

\begin{remark} 
$\mathcal{H}^{s_0+\ell}(M,M)$ is a shape space of order $\ell$ for the action by composition on the left. So is $\D^{s_0+\ell}(M)$ itself.
\end{remark}
\begin{remark} 
While we do not consider fractional order for shape spaces or  diffeomorphisms of fractional Sobolev regularity out of a desire to keep things more simple, all definitions and results can be straightforwardly extended to those cases.
\end{remark}

\begin{remark}The following properties immediately follow from the definition.
\begin{itemize}
\item For every $\varphi\in\mathcal{D}^{s+k}(M)$ and $X\in \Gamma^{s+k}(TM),$ the mappings $q\mapsto \varphi\circ q$ and $q\mapsto\xi_q(X)$ are of class $\mathcal{C}^{k}$. Moreover, $q\mapsto \xi_q$ is of class $\mathcal{C}^{k-1}$ (see \cite{O} for example).
\item A shape space of order $\ell$ is also a shape space of every order $\ell'>\ell$.
\end{itemize}
\end{remark}

\begin{definition}
We say that $q\in\mathcal{S}$ has \textbf{compact support} if there exists a compact subset $U$ of $M$ such that $R_q:\varphi\mapsto\varphi\cdot q$ is continuous with respect to the semi-distance $d_{H^{s_0+\ell}(U,M)}$ on $\D^s(M)$.
\end{definition}
In other words, $q$ has a compact support if $\varphi\cdot q$ depends only on the restriction of $\varphi$ to a compact subset $U$ of $M$.

\begin{example}\label{example1}
Here are some examples of the most widely used shape spaces:
\begin{enumerate}

\item Let $N$ be a smooth compact Riemannian manifold, and $\alpha_0$ be the smallest integer greater than $\dim(N)/2$. Then $\s=\mathrm{Emb}^{\alpha_0+\ell}(N,M)$ and $\mathrm{Emb}_{\ell}(N,M)$,  the manifolds of all embeddings $q:N\rightarrow M$ respectively of class $H^{\alpha_0+\ell}$ and $\mathcal{C}^\ell$, $\ell\in\N$, are shape spaces of order $\max(1,\ell)$. In this case, $\mathcal{D}^{s_0+\max(1,\ell)}(M)$ acts on $\mathcal{S}$ by left composition $\varphi\cdot q=\varphi\circ q$, and this action satisfies all the required properties of Definition \ref{shapespace} (see \cite{ATY} for the proof), with infinitesimal action $\xi_qX=X\circ q$. Every $q\in\s$ has compact support.

\item A particularly interesting case is obtained when $\dim(S)=0$. Then $N=\{a_1,\dots,a_n\}$ is simply a finite set. In that case, for any $\ell$, the shape space $\s=\mathcal{C}^\ell(S,M)$ is identified with the space of $n$ landmarks in $M$: 
$$
\mathrm{Lmk}_n(M)= \lbrace (x_1,\dots,x_n)\in M^n\ \vert\ x_i\neq x_j\ \text{if}\ i\neq j\rbrace.
$$
For $q=(x_1,\dots,x_n)$, the action of $\mathcal{D}^{s_0+1}(M)$ is given by $\varphi\cdot q=(\varphi(x_1),\dots,\varphi(x_n))$.
For a vector field $X$ of class $H^{s_0+1}$ on $M$, the infinitesimal action of $X$ at $q$ is given by $\xi_q(X)=(X(x_1),\dots,X(x_n))$. Spaces of landmarks are actually spaces of order $0$ (also see \cite{ATY}). 

\item Let $\mathcal{S}_1$ and $\mathcal{S}_2$ be shape spaces of respective orders $\ell_1$ and $\ell_2$. Then the space $\mathcal{S}_1\times \mathcal{S}_2 $ is a shape space of order $\max(\ell_1,\ell_2)$ for the diagonal action $\varphi\circ(q_1,q_2)=(\varphi\cdot q_1,\varphi\cdot q_2)$.

\item Let $\mathcal{S}$ be a shape space of order $\ell\in\N$. Then $T\mathcal{S}$ is a shape space of order $\ell+1$ for the action of $\mathcal{D}^{s_0+\ell+1}(M)$ on $T\mathcal{S}_1$ defined by $\varphi\cdot (q,v)=(\varphi\cdot q,\partial_q (\varphi\cdot q)(v))$.

\end{enumerate}
\end{example}

\begin{remark}
The first example above implies that $\mathrm{Emb}^{\infty}(\s,M)=\cap_{\ell}\mathrm{Emb}^{\alpha_0+\ell}(\s,M)$ is an ILH-manifold, on which $\mathcal{D}^{\infty}(M)$ acts smoothly as an ILH Lie group \cite{O}. Such a space can be called a shape space of infinite order.  They will be briefly studied in Section \ref{sec:equivgeod}.
\end{remark}

\paragraph{Deformation of a state.} A \textit{deformation} of a state $q_0$ is a curve $q(\cdot):[0,1]\rightarrow\s$ of the form
$$
\forall t\in[0,1],\quad q(t)=\varphi(t)\cdot q_0,
$$
where $\varphi(\cdot):[0,1]\rightarrow\mathcal{D}^s(M)$ and $\varphi(0)=e$. Such a $\varphi(\cdot)$ is called a \textit{lift} of the deformation $q(\cdot)$. Now if $\varphi(\cdot)$ is the flow of some vector field $X(\cdot)\in L^2(0,1;\Gamma^s(TM))$, the deformation $q(\cdot)=\varphi(\cdot)\circ q_0$ belongs to $H^1(0,1;\mathcal{S})$ and is the unique solution to the Cauchy problem
\begin{equation}\label{eqxi}
\dot q(t) = \xi_{q(t)}X(t) = X(t)\circ q(t),\quad q(0)=q_0,
\end{equation}
for almost every $t\in[0,1]$. In other words, using the notations of 4 in Definition \ref{shapespace}, for every $X\in L^2(0,1;V),$ and $t\in[0,1]$,
$$
q^X(t)=\varphi^X(t)\cdot q_0.
$$

\subsection{Sub-Riemannian structure on shape spaces}
Let $\mathcal{S}$ be a shape space of order $\ell\geq1$ in $M$, and fix $k\in \N\setminus\{0\}$. Consider $(V,\langle\cdot,\cdot\rangle)$ an arbitrary Hilbert space of vector fields with continuous inclusion in $\Gamma^{s_0+\ell+k}(TM)$. According to Section \ref{sec_SRstructure}, we obtain a strong right-invariant sub-Riemannian structure $\mathrm{SR}_V(s)$ of class $\mathcal{C}^k$ induced by $V$ on $\mathcal{D}^s(M)$, with $s=s_0+\ell$. 

\paragraph{The framework of shape and image matching.} The classical LDDMM algorithms for exact shape matching seek to minimize 
$$
\frac{1}{2}\int_0^1\left<X(t),X(t)\right>dt
$$
over every $X\in L^2(0,1;V)$ such that $\varphi^X(1)\cdot q_0=q_1$, where the template $q_0$ and the target $q_1$ are fixed. Usually, one only wants to get "close" to the target shape, which is accomplished by minimizing
$$
\frac{1}{2}\int_0^1\left<X(t),X(t)\right>dt+g(\varphi^X(1)\cdot q_0)
$$
over every $X\in L^2(0,1;V)$, where the endpoint constraint has been replaced with the addition of a data attachment term $ g(\varphi^X(1)\cdot q_0)$ in the functionnal \cite{GTY1,CTr}. The function $g$ is usually such that it reaches its minimum at $q_1$.

\paragraph{The sub-Riemannian structure.} The previous discussion leads us to define a sub-Riemannian structure on $\s$ as follows.

\begin{definition}\label{def:sr}
The strong sub-Riemannian structure induced by $V$ on $\s$ is the triple 
$$
\mathrm{SR}_V^\s=(\s\times V,\langle\cdot,\cdot\rangle,\xi),
$$
where $\xi$, in a slight abuse of notation, denotes the restriction to $\s\times V$ of the infinitesimal action $\xi:\s\times \Gamma^s(TM)\rightarrow T\s$. This structure is of class $\mathcal{C}^k$. Horizontal systems are couples 
$$
(q(\cdot),X(\cdot))=H^1(0,1;\s)\times L^2(0,1;V)
$$ 
such that  for almost every $t\in[0,1]$,
$$
\dot{q}(t)=\xi_{q(t)}X(t).
$$
The curve $q(\cdot)$ is called a \textbf{horizontal deformation} of $q(0)$. 
\end{definition}
\begin{remark}
If $\xi_q(V)=T_q\s$ for every $q\in\s$, this is actually a Riemannian structure. This is often the case in numerical simulations, where $\s$ is finite dimensional (usually a space of landmarks). However, in the general case, we \emph{do not} obtain a Riemannian structure.
\end{remark}
The \emph{length} and \emph{action} of a horizontal system $(q(\cdot),X(\cdot))$ only depend on $X(\cdot)$ and are respectively given by
$$
\begin{aligned}
L(q(\cdot),X(\cdot))&=L(X(\cdot))=\int_0^1\sqrt{\langle X(t),X(t)\rangle} dt,
\\
 A(q(\cdot),X(\cdot))&=A(X(\cdot))=\frac{1}{2}\int_0^1\langle X(t),X(t)\rangle dt.
\end{aligned}
$$
The LDDMM algorithm can therefore be formulated as a search for sub-Riemannian geodesics on $\s$ for the structure $\mathrm{SR}_V^\s$.

A horizontal system $(q(\cdot),X(\cdot))$ satisfies $q(\cdot)=\varphi^X(\cdot)\cdot q(0)$. Moreover, $\varphi^X(\cdot)$ is horizontal for $\mathrm{SR}_V(s)$ on $\D^s(M)$. The flow $\varphi^X(\cdot)$ is called a \textit{horizontal lift} of the curve $q(\cdot)$. This explains the term \textit{horizontal deformation} for $q(\cdot)$.

\subsection{Minimal lift of a horizontal curve.}\label{sec:minlift}

The length and action of a horizontal system $(q(\cdot),X(\cdot))\in \Omega^\s_V$ coincide with the sub-Riemannian length and action of the horizontal lift $\varphi^X(\cdot)$ on $\D^s(M)$. However, there may be two distinct controls $X_1(\cdot)\neq X_2(\cdot)$ inducing the same horizontal curve $q(\cdot)$. For example, if $\s=Lmk_n(M)$, and $q_0=(x_1,\dots,x_n)\in \s$, then the flow of any $X(\cdot)\in L^2(0,1;V)\setminus\{0\}$ such that 
$$
X(t,x_i)=0,\quad i=1,\dots,n,\  \ t\in[0,1],
$$
leaves $q_0$ fixed, and induces the constant deformation $t\mapsto q_0$. Since different controls can have different actions, it is natural to ask whether, given a horizontal deformation $q(\cdot)$, there exists a control $X(\cdot)$ with minimal action such that $q(\cdot)=\varphi^X(\cdot)\cdot q(0)$.

To answer this question, fix $q\in \s$. Then $\xi_q$ has a closed null set $\ker \xi_q\subset V$. $V$ is Hilbert, so $\ker\xi_q$ admits a unique orthogonal supplement $(\ker\xi_q)^\perp$ such that the restriction of $\xi_q$ to $(\ker\xi_q)^\perp$ is a bijection onto $\xi_q(V)$. The inverse $\xi_q^{-1}:\xi_q(V)\rightarrow (\ker\xi_q)^\perp$ of this bijection, denoted $\xi_q^{-1}$ in a slight abuse of notation, is the \textit{pseudoinverse} of $\xi_q$. This means that $\xi_q\xi_q^{-1}=\mathrm{id}_{\xi_q(V)}$, but that $\xi_q^{-1}\xi_q$ is the orthogonal projection $V\rightarrow (\ker\xi_q)^\perp$. 

Now, for every $w\in\xi_qV$ and $X\in V$ with $\xi_qV=w$, we have
\begin{equation}
\label{eq:mliftvec}
 \langle X,X\rangle=\min_{Y\in V,\ \xi_qY=w}(\langle Y,Y\rangle)\ \Longleftrightarrow\ X=\xi_q^{-1}w.
\end{equation}
We immediately obtain the following lemma.
\begin{lemma}\label{lem:minlift}
A horizontal curve $q(\cdot)\in H^1(0,1;\s)$ admits a unique control $X(\cdot)$  with minimal action $A(X(\cdot))$ such that $\dot{q}(t)=\xi_{q(t)}X(t)$, a.e. $t\in[0,1]$. This control is given for almost every $t\in[0,1]$ by
$$
X(t)=\xi_{q(t)}^{-1}\dot{q}(t).
$$
The corresponding horizontal lift $\varphi^X(\cdot)$ is called the \textbf{minimal lift} of $q(\cdot)$, and it is the horizontal curve on $\D^s(M)$ with minimal action such that
$$
q(\cdot)=\varphi^X(\cdot)\cdot q(0).
$$
\end{lemma}
\begin{remark}\label{rk:minlift} One can then define the length and action of a horizontal curve $q(\cdot)$ independently of a corresponding control by
$$
\begin{aligned}
L^\s(q(\cdot))&=\int_0^1\sqrt{\langle \xi_{q(t)}^{-1}\dot{q}(t),\xi_{q(t)}^{-1}\dot{q}(t)\rangle} dt,
\\
 A^\s(q(\cdot))&=\frac{1}{2}\int_0^1\langle \xi_{q(t)}^{-1}\dot{q}(t),\xi_{q(t)}^{-1}\dot{q}(t)\rangle dt.
\end{aligned}
$$
However, $L^\s$ and $A^\s$ are much harder to handle than $A$ and $L$ (for example, they may not be differentiable), so we will not use them.
\end{remark}

\paragraph{The dual viewpoint.} Computing the minimal lift of a given horizontal curve $q(\cdot)$ is very hard to do, as we would need to compute the pseudo-inverse $\xi_q^{-1}$. However, one can use a dual point of view to directly generate controls $X(\cdot)$ that correspond to minimal lifts.


\begin{lemma}\label{lem:Kq}
Let $q\in M$ and $p\in T_q^*M$. The \textbf{momentum map} of the action of $\mathcal{D}^s(TM)$ on $\mathcal{S}$ is defined by $\mu^\s:(q,p)\mapsto \xi_q^*p=p\diamond q\in \Gamma^s(TM)^*$. \\
Then $X=K_{V}\xi_q^*p$ belongs to $(\ker \xi_q)^\perp$. In particular,
$$
X=\mathrm{argmin}\{\langle Y,Y\rangle\mid Y\in V,\ \xi_qY=\xi_qX\}.
$$
\end{lemma}
 \begin{proof}
Let $X=K_V\xi_q^*p$ (recall that $K_V$ is the canonical isometry $V^*\rightarrow V$). For every $X_0\in\ker(\xi_q)$, we have $\langle X,X_0\rangle=\langle K_{\mathcal{H}}^e\xi_q^*p,X_0\rangle=\xi_q^*p(X_0)=p(\xi_qX_0)=0$, hence $X\in\ker(\xi_q)^\perp$.
\end{proof}

We define the operator 
$$
K_q=\xi_qK_V\xi_q^*:T_q^*\s\rightarrow \xi_q(V)\subset T_q\s.
$$
$K_q$ is symmetric (i.e., $p_1(K_qp_2)=p_2(K_qp_1)$), positive semi-definite (i.e., $p(K_qp)\geq 0$), and we have $\langle K_{V}\xi_q^*p,K_{V}\xi_q^*p\rangle=p(K_qp)$. The map $(q,p)\mapsto K_qp$ is of class $\mathcal{C}^{k}$. Moreover $\xi_q(V)=K_q(T_q^*\s)$ when  $\xi_q(V)$ is closed in $T_q\s$ \cite{BBOOK}.

Then curves $q(\cdot)\in H^1(0,1;\s) $ satisfying
$$
\dot{q}(t)=K_{q(t)}p(t),\quad p(t)\in T_{q(t)}^*\s,
$$
almost everywhere are horizontal, and have minimal lift $X(t)=K_V\xi_{q(t)}^*p(t)$, with action
$$
\frac{1}{2}\int_0^1\left< X(t),X(t)\right>\d t=\frac{1}{2}\int_0^1p(t)\left( K_{q(t)}p(t)\right)\d t.
$$
\paragraph{Examples using the reproducing kernel.} As mentioned in Remark \ref{rkhs}, if $V$ is a Hilbert space of vector fields on $M$ with continuous inclusion in $\Gamma^s(TM)$, then $V$ is uniquely characterized by its reproducing kernel $K$, which is a section of the vector bundle $\mathrm{L}(T^*M,TM)\rightarrow M\times M$ such that $\int_{M\times M}P(x)\left(K(x,y)P(y)\right)\,\d y\,\d x \geq 0$ for every continuous one-form $P:M\rightarrow T^*M$.
More generally, this kernel also converts any continuous linear form $P$ on $V$ into the vector field $X\in V$ such that $P=\langle X,\cdot\rangle$, by convolution of $P$ with $K$: $X(x)=K(x,\cdot)P=\int_{M}K(x,y)P(y)\,\d y$.

In shape analysis, it is common to start with an explicit kernel $K$ instead of the space $V$. 
\begin{example} A commonly used example is the Gaussian kernel when $M=\R^d$:
$$
K(x,y)p=e^{-\frac{\vert x-y\vert^2}{\sigma^2}}p^T,
$$
where we identified linear forms $p\in (\R^d)^*$ with lign matrices, and vectors in $\R^d$ with column matrices.
\end{example}
Such a kernel does generate a unique Hilbert space of vector fields $V$, but the space itself is not explicit. Consequently, it is hard to say if a given vector field belongs to $V$. On the other hand, it is easy to generate explicit examples of elements of $V$: since we have an explicit formula for the reproducing kernel, we can simply take convolutions of continuous 1-forms (or, more generally, 1-cocurrents) with $K$. This is another reason why the dual viewpoint is so important.

\begin{example}
Take $\mathcal{S}=\mathrm{Lmk}^n(M)$. Any element $q$ of $\mathcal{S}$ is a $n$-uple $q=(x_1,\dots,x_n)$ with $x_i\in M$, and the infinitesimal action of $X\in \Gamma^s(TM)$ is $X\circ q=\xi_qX=(X(x_1),\dots,X(x_n))$.
Any one-form on $\mathcal{S}$ is a $n$-tuple $p=(\alpha_1,\dots,\alpha_n)$ with $\alpha_i\in T^*_{x_i}M$, and
$p(w)=\sum_{i=1}^n \alpha_i(w_i)$.
We have $p\diamond q(X)=p(X\circ q)=\sum_{i=1}^n \alpha_i(X(x_i))$, and hence $p\diamond q=\sum_{i=1}^n \alpha_i\otimes \delta_{x_i}$, with $\delta_x$ the usual Dirac mass. Therefore, if $(x,y)\mapsto K(x,y)$ is the reproducing kernel of the Hilbert space $V$, we get $K_V\xi_q^*p(x)=\sum_{i=1}^n K(x,x_i)\alpha_i$, and
$$
K_qp=\xi_qK_V\xi_q^*p=\sum_{i=1}^n \left( K(x_1,x_i)\alpha_i,\dots,K(x_n,x_i)\alpha_i \right)\in T_{x_1}M\times\dots\times T_{x_n}M.
$$
In particular,
$$
\left<K_V\xi_q^*p,K_V\xi_q^*p\right>=p(K_qp)=
\sum_{i,j=1}^np_i(K(x_i,x_j)p_j).
$$
\end{example}

\begin{example}
Take $\mathcal{S}=\mathcal{C}^{0}(N,M)$, with $N$ a compact Riemannian manifold. An element $q$ of $\mathcal{S}$ is a continuous mapping $q:N\rightarrow M$, and 
$$
T_q\s=\Gamma_0(q^*TM)=\{v\in\mathcal{C}^0(N, TM)\mid \forall a\in S,\ v(a)\in T_{q(a)}M\}
$$
The infinitesimal action of $X\in V$ at $q$ is
$$
\xi_qX=X\circ q\in \Gamma_0(q^*TM).
$$
A one-form $p\in T^*_q\mathcal{S}$ is a section of $q^*T^*M$ with (distributional) coefficients in the space of Radon measures. Then
$$
p(v)=\int_{N}dp(a)(v(a)).
$$
Now for $X\in V$, we have
$$
p(\xi_qX)=\int_{N}dp(a)(X(q(a))),
$$
so for $x\in M$,
$$
K_V\xi_q^*p(x)=\int_NK(x,q(a))dp(a).
$$
For $K_qp:N\rightarrow TM$, we obtain
$$
K_qp(a)=\int_NK(q(a),q(a'))dp(a')
$$
and
$$
\left<K_V\xi_q^*p,K_V\xi_q^*p\right>=p(K_qp)=\iint_{N^2}dp(a)(K(q(a),q(a')))
dp(a').
$$
\end{example}

\subsection{Sub-Riemannian distance}

We keep the same notations as in the previous section. The sub-Riemannian distance induced by the sub-Riemannian structure $\mathrm{SR}_V^\s=(\s\times V,\langle\cdot,\cdot\rangle,\xi)$ is denoted $d_{SR}^\s:\s\times\s\rightarrow\R_+\cup\{\infty\}$. Recall that $d_{SR}^\s(q_0,q_1)$ is the infimum over the lengths of every horizontal system $(q(\cdot),X(\cdot))$ with $q(0)=q_0$ and $q(1)=q_1.$ 
It is clear that $d_{SR}^{\s}$ is at least a pseudo-distance (it satisfies all axioms of a distance function except, possibly, the separation axiom).

\medskip

Note that we have
\begin{equation}
\label{eq:srdiffsh}
d_{SR}^\s(q_0,q_1)=\inf_{\varphi,\ \varphi\cdot q_0=q_1} \left( d_{SR}(e,\varphi)\right)=d_{SR}(e,R_{q_0}^{-1}(\{q_1\})),
\end{equation}
with $\inf(\emptyset)=+\infty$ and $R_{q_0}^{-1}(\{q_1\})=\{\varphi\in\D^s(M)\mid \varphi\cdot q_0=q_1\}$. Indeed, if $d_{SR}(e,\varphi)<+\infty$ and $\varphi\cdot q_0=q_1$, we can consider $X(\cdot)$ such that $\varphi^X(\cdot)$ is the minimizing geodesic between $e$ and $\varphi$. But then $d_{SR}(e,\varphi)=L(X(\cdot))\geq d_{SR}^\s(q_0,q_1)$. On the other hand, if $(X_n(\cdot))_{n\in\N}$ is a sequence of controls steering $q_0$ to $q_1$ with $L(X_n(\cdot))\rightarrow d_{SR}^\s(q_0,q_1)$ as $n$ goes to infinity, then $d_{SR}^\s(e,R_{q_0}^{-1}(\{q_1\}))\leq d_{SR}(e,\varphi^{X_n}(1))\leq L({X_n}(\cdot))$.


\medskip

A horizontal system $(q(\cdot),X(\cdot))$ is said to be \textit{minimizing} if it minimizes the action among all other systems with the same endpoints. The horizontal curve $q(\cdot)$ is called a \textit{minimal curve}, while $X(\cdot)$ is a \textit{minimal control}. In this case, Cauchy-Schwarz's theorem tells us that
$$
d_{SR}^\s(q(0),q(1))=L(q(\cdot),X(\cdot))=\sqrt{2A(q(\cdot),X(\cdot))}.
$$ 
Note that the minimal lift $\varphi(\cdot)$ of a minimizing curve $q(\cdot)$ is a minimizing curve between $e$ and $\varphi(1)$. In this case, $d_{SR}^\s(q(0),q(1))=d_{SR}(e,\varphi(1))$.

\begin{theorem}\label{th:complete2}
The sub-Riemannian distance is a true distance, that is, $d_{{SR}}^\s(q_0,q_1)=0$ implies $q_0=q_1$. It induces a topology that is at least as fine as the manifold topology of $\s$. \\
Moreover, if $q_0\in \s$ has compact support, and if we denote $\s_{q_0}=\D^s(M)\cdot q_0$, then $(\s_{q_0},d_{SR}^{\s})$ is geodesically complete (any two points within finite distance can be connected by a minimizing geodesic).
\end{theorem}

\begin{proof} Since Banach manifolds are Hausdorff spaces, proving that $d_{SR}^{\s}$ induces a finer topology on $\s$ will also prove that it is a true distance. For this, let $U\subset\s$ be an open subset for the manifold topology. Fix $q\in U$. Then $R_q^{-1}(U)$ is open in $\D^s(M)$ and contains $e$.  Theorem \ref{th:complete} implies that there exists $r>0$ such that the sub-Riemannian ball $B$ in $\D^s(M)$ of center $e$ included in $R_q^{-1}(U)$. But we saw earlier in the section that $R_q(B)$ is exactly equal to the sub-Riemannian ball on $\s$ of center $q$ with same radius as $B$. Therefore, $U$ contains an open ball for $d_{SR}^{\s}$, and the sub-Riemannian topology is finer than the manifold topology.

Now let us prove the geodesic completeness of our space for compactly supported states. Assume $q_0$ has compact support. It is clear that any element of $\s_{q_0}$ also has compact support. Hence, to prove that $\s_{q_0}$ is geodesically complete, we just need to prove that if $d^\s(q_0,q_1)<+\infty$, then $q_0$ and $q_1$ can be joined by a minimizing geodesic.

Let $(q^n(\cdot))_{n\in\N}$ be a minimizing sequence of horizontal curves in $H^1(0,1;\mathcal{S})$ such that $q^n(0)=q_0$ and $q^n(1)=q_1$ for every $n$, induced by a sequence of controls $X^n(\cdot)$, such that $\sqrt{2A(q^n(\cdot),X^n(\cdot))}$ converges to $d_{SR}^{\mathcal{S}}(q_0,q_1)$. For every $n$, denote by $\varphi^n(\cdot)$ the flow of $X^n$.
Up to the extraction of some subsequence, we have $X^n(\cdot)\rightarrow \bar{X}(\cdot)$ for the weak topology in $L^2(0,1;V)\subset L^2(0,1;\Gamma^{s_0+\ell+1}(TM))$. Let $\bar\varphi(\cdot)=\varphi^{\bar X}(\cdot)$.  Lemma 2 from \cite{AT} implies that $\varphi^n(1)_{\vert U}\rightarrow\bar\varphi(1)_{\vert U}$ strongly in $H^s$ as $n\rightarrow+\infty$ for every compact subset $U$ of $M$. Since $q_0$ has a compact support, this implies that the horizontal curve $\bar{q}(\cdot)=\bar\varphi(\cdot)\cdot q_0$ satisfies $\bar{q}(1)=q_1$. But
$$
A(\bar{X}(\cdot))\leq \liminf_{n\rightarrow\infty}A(X^n(\cdot))
=\frac{1}{2}d_{SR}^{\mathcal{S}}(q_0,q_1)^2.
$$
%
%
%
\end{proof}

\subsection{Sub-Riemannian geodesics and singular curves on shape spaces}\label{geodshsp}

We assume that $\mathcal{S}$ is a shape space in $M$ of order $\ell\geq1$, and that $\mathcal{D}^s(M)$, $s=s_0+\ell$, is equipped with a strong right-invariant sub-Riemannian structure induced by the Hilbert space $(V,\left<\cdot,\cdot\right>)$ of vector fields on $M$, with continuous inclusion $V\hookrightarrow \Gamma^{s+k}(TM)$ for some $k\geq 1$.

\paragraph{The differential structure of the space of horizontal systems and the endpoint mapping.}  Without further assumptions on $\s$ (for example, the existence of a local addition), it is not known whether $H^1(0,1;\s)$ admits a natural differential structure \cite{KMBOOK}.  However, Condition 4 of Definition \ref{shapespace} shows that the space of all horizontal systems with starting point $q_0$ is in one-to-one correspondance with $ L^2(0,1;V)$ through the identification $X(\cdot)\leftrightarrow (\varphi^X(\cdot)\cdot q_0,X(\cdot))$. This endowes the space of horizontal systems from $q_0$ with a smooth manifold structure with a single coordinate chart. We will always identify it with $ L^2(0,1;V)$ through this correspondance from now on.

\paragraph{Geodesics.} Fix an initial point $q_0$ and a final point $q_1$ in $ \mathcal{S}$. The endpoint mapping from $q_0$ is 
$$
\mathrm{end}^\s_{q_0}(X(\cdot))=\varphi^X(1)\cdot q_0=R_{q_0}\circ \mathrm{end}(X(\cdot)),
$$
where $\mathrm{end}(X(\cdot))=\varphi^X(1)\in\D^s(M)$. It is of class $\mathcal{C}^{k}$, as a composition of $\mathrm{end}(X(\cdot))$ and $R_{q_0}$, which is of class $\mathcal{C}^k$ on Im$(\mathrm{end})\subset\mathcal{D}^{s+k}(M)$. A \textit{geodesic} on $\s$ between the states $q_0$ and $q_1$ is a horizontal system $(q(\cdot),X(\cdot))$ joining $q_0$ and $q_1$ that is a critical point of the action among all horizontal systems with the same endpoints, i.e., a critical point of $A(X(\cdot))=A(q(\cdot),X(\cdot))$ restricted to $\mathrm{end}_{q_0}^{-1}(\{q_1\})$. In other words, for any $\mathcal{C}^1$-family $a\mapsto X_a(\cdot)\in L^2(0,1;V)$ such that $\varphi^{X_a}(1)\cdot q_0=q_1$ for every $a$, with $X_0=X$, we have
$$
\partial_aA(X_a(\cdot))_{\vert a=0}=0.
$$
\begin{remark}
If $(q(\cdot),X(\cdot))$ is a geodesic, then $\varphi^X$ is the minimal lift of $q(\cdot)$. Indeed, if $X(t)$ does not belong almost everywhere to $\h_q=(\ker \xi_q)^\perp$, it admits a unique split $X=X_0+X_1$, with $X_0(t)\in\ker \xi_{q(t)}$ almost everywhere and $X_0(\cdot)\neq 0$. But then $X_s=X-sX_0$ generates the same horizontal curve $t\mapsto q(t)$ and 
$$\partial_s A(X_s)=-\int_0^1\langle X_0(t),X_0(t)\rangle dt\neq0.$$
Consequently, a geodesic $(q(\cdot),X(\cdot))$ is uniquely determined by its trajectory $q(\cdot)$, and we can identify geodesics with trajectories of geodesics. We will do so in the rest of the section.
\end{remark}

\begin{remark}
Any sub-arc of a geodesic is also a geodesic (up to a reparametrization).
\end{remark}

In Section \ref{sec2}, we saw that there can be three kinds of geodesics on $\D^s(M)$, depending on the image of the derivative of the couple $(A,\mathrm{end})$. This is also true here, at least in the general case where $\s$ has infinite dimensions. If $(q(\cdot),X(\cdot))$ is a geodesic, that is, a critical point of the action with fixed endpoints, then one of two possibilities is satisfied:
\begin{enumerate}
\item There exists $(\lambda,p_1)\in \R\times T_{q_1}^*\s\setminus\{0,0\}$ such that
\begin{equation}\label{eq:lagcondshsp}
\d\,\mathrm{end}^\s_{q_0}(X(\cdot))^*p_1+\lambda \d A(X(\cdot))=0.
\end{equation}
This is a Lagrange mutiplier characterization, which splits into two subcases:
\begin{itemize}
\item If $\lambda\neq0$, we can take $\lambda=-1$ so that $ \d A(X(\cdot))=\d\,\mathrm{end}^\s_{q_0}(X(\cdot))^*p_1$, and say that $q(\cdot)$ is a \textit{normal geodesic}.
\item If $\lambda=0$, $p_1\neq0$ and $\d\,\mathrm{end}^\s_{q_0}(X(\cdot))^*p_1=0$, in which case $X(\cdot)$ is a singular point of $\mathrm{end}^\s_{q_0}$. We say that $q(\cdot)$ is an \textit{abnormal geodesic.}
\end{itemize}
\item The image of $(\d A(X(\cdot)),\d\,\mathrm{end}^\s_{q_0}(X(\cdot)))$ is a proper dense subset of $\R\times T_{q_1}\s$: there is no Lagrange multipliers characterization. We say that $q(\cdot)$ is an \textit{elusive geodesic.}
\end{enumerate}
Indeed, for a geodesic, the couple of linear maps $(\d A(X),\d\,\mathrm{end}^\s_{q_0}(X))$ can't be surjective. Then, if its image is not a proper dense subset (i.e., if we are not in the second case), it is included in a closed hyperplane, and we are in the first case.

\subsubsection{Statement of the results} 

In this section, we define the Hamiltonian and the corresponding geodesic flow.

\paragraph{Canonical symplectic form, symplectic gradient.} We denote by $\omega$ the canonical weak symplectic form on $T^*\s$, given by the formula
$$
\omega(q,p).(\delta q_1,\delta p_1;\delta q_2,\delta p_2)=\delta p_2(\delta q_1)-\delta p_1(\delta q_2),
$$
with $(\delta q_i,\delta p_i)\in T_{(q,p)}T^*\s\simeq T_q\s\times T^*_q\s$ in a canonical coordinate system $(q,p)$ on $T^*M$. 

A function $f:T^*\s\rightarrow \R$, differentiable at some point $(q,p)\in T^*\s,$ admits a \textit{symplectic gradient} at $(q,p)$ if there exist a vector $\nabla^\omega f(q,p)\in T_{(q,p)}T^*\s$ such that, for every $z\in  T_{(q,p)}T^*\s,$
$$
df_{(q,p)}(z)=\omega(\nabla^\omega f(q,p),z).
$$
In this case, this symplectic gradient $\nabla^\omega f(q,p)$ is unique.

Such a gradient exists if and only if $\partial_pf(q,p)\in T_q^{**}\s$ can be identified with a vector in $T_q \s$ through the canonical inclusion $T_q \s\hookrightarrow T^{**}_q\s.$ In that case, we have, in canonical coordinates,
$$
\nabla^\omega f(q,p)=(\partial_pf(q,p),-\partial_qf(q,p)).
$$

\paragraph{The normal Hamiltonian function and geodesic equation.} We define the \emph{normal Hamiltonian} of the system $H^\s:T^*\s\rightarrow\R$ by
$$
H^\s(q,p)=\frac{1}{2}p(K_qp)=\frac{1}{2}p(\xi_qK_V\xi_q^*p).
$$
This is a function of class $\mathcal{C}^k$. Moreover, for every $(q,p)\in T^*\s$ and $\delta p\in T^*_q\s$, $\partial_p H^\s(q,p)(\delta p)=\delta p(K_qp),$ so that we can identify $\partial_p H^\s(q,p)\in T^{**}_q\s$ with $K_qp\in T_qM$: $\nabla^\omega H^\s(q,p)$ is well-defined on $T^*\s$, and given in canonical coordinates by
$$
\nabla^\omega H^\s(q,p)=(K_qp,-\frac{1}{2}\partial_q(K_qp)^*p).
$$
Note that $\nabla^\omega H^\s$ is of class $\mathcal{C}^{k-1}$ on $T^*\s$. Hence, if $k\geq2$, it admits a unique local flow, called the Hamiltonian geodesic flow.
\begin{theorem}\label{theo:carachamilnorm}
Assume $k\geq1$. Then a horizontal curve $q(\cdot)$ is a geodesic if and only if it is the projection of an integral curve $(q(\cdot),p(\cdot))$ of $\nabla^\omega H$. That is, $(\dot{q}(t),\dot{p}(t))=\nabla^\omega H^\s(q(t),p(t))$ for almost every $t$ in $[0,1]$.

\medskip

Assume $k\geq2$. Then $\nabla^\omega H$ admits a \textbf{global flow} on $T^*\s$ of class $\mathcal{C}^{k-1}$, called the \textbf{Hamiltonian geodesic flow}. In other words, for any initial point $(q_0,p_0)\in T^*\s$, there exists a unique curve $t\mapsto (q(t),p(t))$ defined on all of $\R$, such that $(q(0),p(0))=(q_0,p_0)$ and, for almost every $t$,
$$
(\dot{q}(t),\dot{p}(t))=\nabla^\omega H^\s(q(t),p(t)).
$$
This curve is of class $\mathcal{C}^k$ in time and $\mathcal{C}^{k-1}$ with respect to $(q_0,p_0)$.
\end{theorem}
We say that $p(\cdot)$ is the \textit{normal covector} along the trajectoy.
\begin{remark}
One can then prove that any normal geodesic is, locally, a minimizing curve on small sub-intervals. This is a straightforward generalization of the proof of the same result for finite dimensional sub-riemannian manifolds: use the Hamiltonian flow allows the construction of a calibration of normal geodesics \cite{CTBOOK,MBOOK}.
\end{remark}
This theorem will be proved in Section \ref{sec:proof}, concurrently with Proposition \ref{propo:carachamilsing} found in the next paragraph. Along the way, we will also prove the following result.

\begin{proposition}\label{prop:liftnomalgeod}
Let  $\varphi(\cdot)$ be a horizontal curve from $e$ on $\D^s(M)$. Then $q(\cdot)=\varphi(\cdot)\cdot q_0$ is a normal geodesic on $\s$ if and only if there exists $p_0\in T_{q_0}^*\s$ such that $\varphi(\cdot)$ is a normal geodesic on $\D^s(M)$ with initial normal covector $P_0=\xi_{q_0}^*p_0$.
\end{proposition}

\paragraph{Singular horizontal systems.} Define the \emph{abnormal Hamiltonian} $H_0^\s:T^*\s\times V\rightarrow\R$ by
$$
H^\s_0(q,p,X)=p(\xi_qX)=\xi_q^*p(X).
$$
For fixed $X$, the mapping $(q,p)\mapsto H^\s_0(q,p,X)$ is of class $\mathcal{C}^k$ and $\partial_pH^\s_0(q,p,X)$ can be identified with $\xi_qX$. This mapping therefore admits a symplectic gradient, denoted $\nabla^\omega H^\s_0(q,p,X)\in T_{(q,p)}T^*\s$ and given, in canonical coordinates, by the formula
$$
\nabla^\omega H^\s_0(q,p,X)=(\xi_qX,-\partial_q(\xi_qX)^*p).
$$
This gradient is of class $\mathcal{C}^{k-1}$ in $(q,p)$ and linear in $p$. We have the following result.
\begin{proposition}\label{propo:carachamilsing}
A horizontal system $(q(\cdot),X(\cdot))$ is singular if and only if there exists $p:t\mapsto p(t)\in T_{q(t)}^*\s\setminus\{0\}$ of Sobolev class $H^1$ in time such that, for almost every time $t$ in $[0,1]$, the \textbf{abnormal Hamiltonian equations}
$$
(\dot{q}(t),\dot{p}(t))=\nabla^\omega H^\s_0(q(t),p(t),X(t))=\left(\xi_{q(t)}X(t),-\partial_q(\xi_{q(t)}X(t))^*p(t)
\right),$$
is satisfied, with $p(0)\neq0$ and for almost every time $t$ in $[0,1]$, $\xi_{q(t)}^*p(t)=0$, that is,
$$
\xi_{q(t)}(V)\subset \ker p(t).
$$
In this case, $\d\,\mathrm{end}^\s_{q(0)}(X(\cdot)))^*p(1)=0.$
\end{proposition}
We say that $p(\cdot)$ is a \textit{singular covector} along the trajectory. This will be proved in Section \ref{sec:proof}, along with the following result.

\begin{proposition}\label{prop:liftabnomalgeod}
Let  $\varphi(\cdot)=\varphi^X(\cdot)$ be a horizontal curve from $e$ on $\D^s(M)$. Then the horizontal system $(q(\cdot),X(\cdot))$ with $q(\cdot)=\varphi(\cdot)\cdot q_0$ is singular on $\s$ if and only if there exists $p_0\in T_{q_0}^*\s$ such that $\varphi(\cdot)$ is a singular curve on $\D^s(M)$ with initial singular covector $P_0=\xi_{q_0}^*p_0$.
\end{proposition}
A special case of this result was used in \cite{AT} to give examples of singular curves on $\D^s(M)$.

\subsubsection{Proof of Theorem \ref{theo:carachamilnorm} and Propositions \ref{prop:liftnomalgeod}, \ref{propo:carachamilsing}, and \ref{prop:liftabnomalgeod}} \label{sec:proof}

The purpose of this section is to prove, using the momentum map, that normal geodesics and singular curves on $\s$ are exactly those that can respectively be lifted to normal geodesics and singular curves on $\D^s(M)$. This is a well-known result in the Riemannian framework of Clebsch optimal control \cite{CH,MRBOOK,RV}.

\paragraph{Step 1: Momentum formulation.} Recall that the \textit{momentum map} associated to the group action of $\D^s(M)$ over $\s$ is the mapping $\mu^\s:T^*\s\rightarrow \Gamma^s(TM)^*=\Gamma^{-s}(T^*M)$ given by
$$
\mu^\s(q,p)=p\diamond q=\xi_q^*p.
$$

\begin{proposition}\label{prop:mom1}
Consider a solution $(q(\cdot),p(\cdot))$ to the normal Hamiltonian equation $(\dot{q}(t),\dot{p}(t))=\nabla^\omega H^\s(q(t),p(t))$ and let $X(\cdot)=K_V\xi_{q(t)}^*p(t)$ and $\mu^\s(t)=\mu^\s(q(t),p(t))$ the momentum along the trajectory. Then, for almost every $t$,
$$
\dot{\mu}^\s(t)=\mathrm{ad}^*_{X(t)}\mu(t)
$$
in the weak sense. In particular, $\varphi^X(\cdot)$ is a normal geodesic on $\D^s(M)$ with initial covector $P(0)=\xi_{q(0)}^*p(0)$ and with momentum $\mu(t)=\mu^\s(t)$. 

The same is true for singular curves: if $p(\cdot)$ is a covector associated to a singular system $(q(\cdot),X(\cdot))$, then $\varphi^X$ is a singular curve in $\D^s(M)$ with singular momentum $\mu(t)=\mu^\s(t)$ and initial singular covector $P(0)=\xi_{q(0)}^*p(0)$.
\end{proposition}

\begin{proof}
Fix $Y\in \Gamma^{s+1}(TM)$. Then
$$
\begin{aligned}
\dot{\mu}^\s(t)(Y)&=p(t)\Big( \partial_q(\xi_{q(t)}(Y))(\xi_{q(t)}(X(t)))- \partial_q(\xi_{q(t)}(X(t)))(\xi_{q(t)}(Y))\Big)\\
&=p(t)([\xi(X(t)),\xi(Y)](q(t)))\\
&=p(t)(\xi_{q(t)}([X(t),Y]))=ad_{X(t)}^*\mu(t)(Y).
\end{aligned}
$$
In the normal case, $t\mapsto \mu(t)$ then satisfies the necessary and sufficient condition for $X(t)=K_V\mu(t)$ to be the control of a Hamiltonian geodesic $t\mapsto(\varphi(t),P(t))$ on $\mathcal{D}^s(M)$ with initial covector $\mu(0)=P(0)$. In the normal case, since $X(t)=\xi_{q(t)}^*p(t)$, $\varphi(\cdot)$ is also the minizing lift of $q$ starting at $e$.

In the singular case, $t\mapsto \mu(t)$ satisfies the necessary and sufficient condition for $X(t)$ to be the control of a singular curve $t\mapsto(\varphi(t),P(t))$ on $\mathcal{D}^s(M)$ with initial covector $P(0)=\mu(0)$.
\end{proof}

We will also need the converse.
\begin{proposition}\label{prop:mom2}
Let $\varphi(\cdot)=\varphi^X(\cdot)$ be a normal geodesic (resp. a singular curve) on $\D^s(M)$ with momentum $\mu(\cdot)$. Assume that $\mu(0)=\xi_{q_0}^*p_0$ for some $p_0\in T^*_{q_0}\s$, and define the horizontal curve $q(\cdot)=\varphi(\cdot)\cdot q_0$. Then there exists $p:t\mapsto p(t)\in T^*_{q(t)}\s$ with $p(0)=p_0$ such that $(q(\cdot),p(\cdot))$ satisfies the normal  (resp. abnormal) Hamiltonian equations. Moreover, $\mu^\s(q(t),p(t))=\mu(t)$ for all $t$.
\end{proposition}
\begin{proof}
Consider $p(\cdot)$ the solution of the linear Cauchy problem $p(0)=p_0$ and
$$
\dot{p}(t)=-\partial_q(\xi_{q(t)}X(t))^*p(t).
$$
Let $\mu^\s(t)=\mu^\s(q(t),p(t))$ for every $t\in[0,1]$. Then, the same proof as in the previous proposition shows that
$$
\dot{\mu}^\s(t)=\mathrm{ad}^*_{X(t)}\mu^\s(t).
$$
Then $\mu^\s(\cdot)=\varphi(\cdot)_*\mu^\s(0)=\varphi(\cdot)_*\mu(0)$, and Proposition \ref{prop_momentum2} implies that $\varphi(\cdot)_*\mu(0)=\mu(\cdot)$. Hence, $\mu(t)=\mu^\s(t)$ for every $t$ in $[0,1]$.

In the normal case, we need to prove that $(q(t),p(t))$ satisfy the normal Hamiltonian equation on $\s$. In the singular case, we already have $\dot{p}(t)=-\partial_qH^\s_0(q(t),p(t),X(t))$, but we must also prove that $\xi_{q(t)}^*p(t)=0$ on $V$.

Let us start with the case where $\varphi$ is a normal geodesic. Then  $X(t)=K_V\mu(t)=K_V\xi_{q(t)}^*p(t)$ for almost every $t$, so $\dot{q}(t)=\xi_{q(t)}X(t)=K_{q(t)}p(t)$. Then 
$$
\begin{aligned}
-\frac{1}{2}\partial_q(p(t)K_{q(t)}p(t))&=-\frac{1}{2}(\partial_q(\xi_{q(t)}K_V\xi_{q(t)}^*p(t))^*p(t)\\
&=
-(\partial_q(\xi_{q(t)})K_V\xi_{q(t)}^*p(t))^*p(t)\\&=-\partial_q(\xi_{q(t)}X(t))^*p(t)
\\&=\dot{p}(t).
\end{aligned}
$$
Hence $(q(t),p(t))$ does satisfy the normal Hamiltonian equation on $\s$.

In the singular case, for any $Y\in V$, $\xi_{q(t)}^*p(t)(Y)=\mu^\s(t)(Y)=\mu(t)(Y)=0$ because $\varphi^X$ is a singular curve with momentum $\mu(\cdot)$, which finishes the proof.
\end{proof}

Combining those results, we see that solutions of the normal (resp. abnormal) Hamiltonian equations on $\s$ are exactly those curves that come from normal geodesics (resp. singular curves) on $\D^s(M)$ with initial momentum of the form $\xi_{q_0}^*p_0$. In particular, for $k\geq2$, the completeness of the the normal geodesic flow on $T^*\D^s(M)$ implies that $\nabla^\omega H^\s$ is a complete vector field on $T^*\s$.

\paragraph{Step 2: Solutions of the Hamiltonian equations are indeed normal geodesics and singular curves.} Recall that if $(q(\cdot),X(\cdot))$ is a normal geodesic (resp. a singular system), then $\d A(X)=\d\,\mathrm{end}_{q_0}^\s (X)^*p_1$ (resp. $0=\d\,\mathrm{end}_{q_0}^\s(X)^*p_1$) for some $p_1\in T_{q(1)}^*\s$. But the following partial converse is trivially true.
\begin{lemma}
Let $(q(\cdot),X(\cdot))$ be a horizontal system with $q(1)=q_1$. Then:
\begin{itemize}
\item If $\d A(X(\cdot))=\d\,\mathrm{end}^\s_{q_0}(X(\cdot))^*p_1$ for some $p_1\in T^*_{q_1}\s$, then $q(\cdot)$ is a {geodesic} (and, therefore, a normal geodesic).
\item If $0=\d\,\mathrm{end}^\s_{q_0}(X(\cdot))^*p_1$ for some nonzero $p_1\in T^*_{q_1}\s$, then $X(\cdot)$ is a singular point of $\mathrm{end}^\s_{q_0}$: $q(\cdot)$ is an \textbf{singular curve.} However, it may not be a geodesic.
\end{itemize}
\end{lemma}
\begin{remark} 
On the other hand, $(\d A(X(\cdot)),\d\,\mathrm{end}_{q_0}^\s(X(\cdot)))$ having proper dense image does not give any useful information on $X(\cdot)$, so the elusive geodesic cannot be found this way.
\end{remark}

Now, note that $\mathrm{end}_{q_0}^\s=R_{q_0}\circ \mathrm{end},$ where $\mathrm{end}:L^2(0,1;V)\rightarrow \D^s(M)$ is the endpoint map from $e$ on $\D^s(M)$: $\mathrm{end}(X(\cdot))=\varphi^X(1)$. Consequently, $\d\,\mathrm{end}_{q_0}^\s(X)^*=\d\,\mathrm{end}(X)^*\d R_{q_0}(\varphi^X(1))^*$. For $\varphi\in\D^s(M)$ and $\delta \varphi=Y\circ\varphi\in T_\varphi \D^s(M)$, we have
$$
\d R_{q_0}(\varphi).\delta \varphi=\xi_{\varphi\cdot q}Y,
$$
so that $\d R_{q_0}(\varphi)^*p=(\d R_{\varphi}^{-1})^*\xi_{\varphi\cdot q}^*p$.

We deduce from the previous discussion that $(q(\cdot),X(\cdot))$ is a normal geodesic (resp. singular curve) between $q_0$ and $q_1$ if and only if $\varphi^X$ is a normal geodesic (resp. singular curve) whose covector $P(\cdot)$ satisfies
$$
P(1)=(\d R_{\varphi}^{-1})^*\xi_{\varphi\cdot q}^*p_1
$$
for some $p_1\in T^*_{q_1}\s$. But $P(1)=(\d R_\varphi^{-1})^*\mu(1)$ with $\mu(\cdot)$ the momentum of $\varphi$, so this condition becomes
$$
\mu(1)=\xi_{\varphi\cdot q}^*p_1=\mu^\s(q_1,p_1).
$$
Proposition \ref{prop:mom2} then proves that there exists $p:t\mapsto p(t)\in T_{q(t)}\s$ such that $(q(t),p(t))$ satisfy the normal (resp. abnormal) Hamiltonian equations, which completes the proof.

\subsubsection{On inexact matching}
It should be emphasized, again, that other geodesics (the so-called \emph{elusive geodesics}) may also exist. Moreover, finding every singular curves and axtracting the abnormal geodesics from them is a daunting task. However, when performing LDDMM methods and algorithms for inexact matching, one aims to minimize over $L^2(0,1;V)$ functionals of the form
$$
J(X(\cdot))=A(X(\cdot))+g(q^X(1))=A(X(\cdot))+
g\circ\mathrm{end}_{q_0}^\s(X(\cdot)).
$$
In this case, $X(\cdot)$ is a critical point if and only if 
$$
dA(X)=-\d\,\mathrm{end}_{q_0}^\s(X)^*dg(q^X(1)).
$$
The trajectory induced by such a critical point $X$ is therefore automatically a normal geodesic, whose covector satisfies $p(1)=-dg(q^X(1))$. This means that one need consider neither abnormal nor elusive geodesics when looking for minimizers of $J$. 

Consequently, the search for minimizing trajectories can be reduced to the minimization of
$$
\frac{1}{2}\int_0^1p(t)(K_{q(t)}p(t))dt+g(q(1))
$$
among all solutions of the control system
$$
\dot{q}(t)=K_{q(t)}p(t),
$$
where $p(\cdot)$ is any covector along $q(\cdot)$ and is $L^2$ in time. 

\begin{remark}
In other words, we reduced ourselves to studying another sub-Riemannian structure $\mathrm{SR}_{V,T^*\s}^{\s}=(T^*\s,g^K,K)$, where $q\mapsto K_q$ is the vector bundle morphism $T^*\s\rightarrow T\s$ and $g_q^K(p_1,p_2)=p_1(K_qp_2).$

This is a \emph{weak} sub-Riemannian structure, whose normal geodesics coincide with those of $\mathrm{SR}_{V}^{\s}$.
\end{remark}

This reduction is very useful in practical applications and numerical simulations, since, when $\s$ is finite dimensional, we obtain a finite dimensional control system, for which many optimization methods are available. See \cite{ATY} for algorithms to minimize such a functionnal in the abstract framework of shape spaces in $\R^d$. 

\subsection{Some examples of geodesic equations on classical shape spaces}\label{ex2}
In this section, we assume that $M=\R^d$.
Let $s_0$ be the smallest integer such that $s_0>d/2$. For every integer $s\geq s_0+1$, the group $\mathcal{D}^s(\R^d)$ coincides with the set of diffeomorphisms $\varphi$ of $\R^d$ such that $\varphi-\mathrm{Id}_{\R^d}\in H^s(\R^d,\R^d)$, and is an open subset of the affine Hilbert space $\mathrm{Id}_{\R^d}+H^s(\R^d,\R^d)$ \cite{BV}. For a vector field $X$, we will also write $X=X^1e_1+\dots+X^de_d=(X^1,\dots,X^d)$, with $(e_i)$ the canonical frame of $\R^d$. 


The Euclidean inner product of two vectors $v$ and $w$ of $\R^d$ is denoted by $v\cdot w$. The notation $v^T$ stands for for the linear form $w\mapsto v\cdot w$. Conversely, for a linear form $p\in(\R^d)^*$, we denote by $p^T$ the unique vector $v$ in $\R^d$ such that $p=v^T$. 

\subsubsection{Spaces of continuous embeddings of compact manifolds and Gaussian kernels.}

We consider the Hilbert space $V\subset H^s(\R^d,\R^d)$ whose reproducing kernel is the Gaussian kernel 
$$
K(x,y)p=e^{-\frac{\vert x-y\vert^2}{2\sigma}}p^T=e(x-y)p^T,
$$
where we denoted $e^{-\frac{\vert x\vert^2}{2\sigma}}=e(x)$ for readability.
We let $N$ be a compact Riemannian manifold, and $\mathcal{S}=\mathcal{C}^0(N,\R^d)$, a shape space of order 1 (and even 0) in $\R^d$. We have $T^*\mathcal{S}=\mathcal{S}\times\mathcal{C}^0(N,\R^d)^*$, with the elements of $\mathcal{C}^0(N,\R^d)^*$ being identified to maps $S\rightarrow(\R^d)^*$ with coefficients in the space $\mathcal{M}(N)$ of Radon measures on $N$, that is, elements of $\mathcal{M}(N)\otimes(\R^d)^*$.

Then, for $q:N\rightarrow \R^d$ continuous and $p=(p_1,\dots,p_d)\in \mathcal{M}(N)\otimes(\R^d)^*$,
$$
(K^V\xi_q^*p)(x)=\int_{N}e(x-q(a))dp^T(a).
$$
Hence, we have $K_qp\in T_q\mathcal{S}=\mathcal{C}^0(N,\R^d)$ given by
$$
(K_qp)(a)=\int_{N}e(q(a)-q(a'))dp^T(a').
$$
The normal Hamiltonian is
$$
H^{\mathcal{S}}(q,p)=\frac{1}{2}\int_{N^2}e(q(a)-q(a'))\sum_{i=1}^d \d p_i(a)\d p_i(a')=\frac{1}{2}\int_{N^2}e(q(a)-q(a'))[\d p(a)\cdot \d p(a')].
$$
Using $e(x-y)=e(y-x)$ and $de(x).v=-\frac{1}{\sigma}(x\cdot v)e(x)$, we get for any $\delta q\in T_q\mathcal{S}$,
$$
\partial_qH^{\mathcal{S}}(q,p).\delta q=-\frac{1}{\sigma}\int_{N^2}[(q(a)-q(a'))\cdot\delta q(a)]e(q(a)-q(a'))[\d p(a)\cdot \d p(a')].
$$
In the end, the geodesic equations read
$$
\begin{aligned}
\dot{q}(t,a)&=\ \ \ \int_{N}e(q(t,a)-q(t,a'))\d p^T(t,a'),\\
\dot{p}(t,a)&=\frac{1}{\sigma}\int_{N}e(q(t,a)-q(t,a'))(q(a)-q(a'))^T[\d p(t,a')\cdot \d p(t,a)].
\end{aligned}
$$

In the special case of landmarks, we can write $\mathcal{S}=(\R^d)^n$, writing $q=(x_1,\dots,x_n)$ and $p=(p_1,\dots,p_n)\in (\R^{d*})^n$, we get (omitting time for readability)
$$
\begin{aligned}
\dot{x}_i&=\sum_{j=1}^n e^{-\frac{\vert x_i-x_j\vert^2}{2\sigma}}p_j^T,\\
\dot{p}_i&=\frac{1}{\sigma}\sum_{j=1}^n e^{-\frac{\vert x_i-x_j\vert^2}{2\sigma}}(p_i\cdot p_j)(x_i-x_j)^T,
\end{aligned}
\qquad \qquad i=1,\dots,n.
$$

\begin{remark}
It is worthwile to note that in the neighbourhood of $q=(x_1,\dots,x_n)$ such that $i\neq j$ implies $x_i\neq x_j$, the structure here is actually Riemannian instead of simply sub-Riemannian.
\end{remark}

\subsubsection{An action of higher order}
In this paragraph, we keep our Gaussian kernel $K$ as above and $N$ our compact Riemannian manifold, but we take the shape space $\mathcal{S}=T\mathcal{C}^0(N,\R^d)=\mathcal{C}^0(N,\R^d)\times\mathcal{C}^0(N,\R^d)$, a shape space of order 1 for the action
$$
\varphi\cdot(q,v)(a)=(\varphi(q(a)),d\varphi(q(a)).v(a)),\quad a\in N.
$$
The three conditions for a shape space of order 1 are easily seen to be true. The infinitesimal action of a vector field $X$ on $\R^d$ is
$$
\xi_{q,v}(X)(a)=(X(q(a)),\d X(q(a)).v(a)),\quad a\in N.
$$

Now a momentum on $\mathcal{S}$ is a couple $(p,l)$ of maps $N\rightarrow(\R^d)^*$ with coefficients in $\mathcal{M}(N)$, and
$$
\xi_{q,v}^*(p,l)(X)=\int_{N}\d p(a)(X(q(a))+\int_{N}\d l(a)(\d X(q(a)).v(a)).
$$
Therefore,
$$
K^V\xi_{q,v}^*(p,l)(x)
=
\int_{N}e(x-q(a))dp^T(a)
-
\frac{1}{\sigma}\int_{N}e(x-q(a))[v(a)\cdot(x-q(a))]\d l^T(a).
$$
Writing $\xi_{q,v}K^V\xi_{q,v}^*(p,l)=(w_1,w_2)$, and $e_{a,a'}=e(q(a)-q(a'))$, we get
$$
\begin{aligned}
w_1(a)=&
\ \ \ \ \ \ \int_{N}e_{a,a'}\d p^T(a')
-
\frac{1}{\sigma}\int_{N}e_{a,a'}[v(a')\cdot(q(a)-q(a'))]\d l^T(a'),
\\
w_2(a)
=
&
-\frac{1}{\sigma}\ \int_{N}e_{a,a'}[v(a)\cdot(q(a)-q(a'))]\d p^T(a')
\\
&
+
\frac{1}{\sigma^2}
\int_{N}e_{a,a'}[v(a')\cdot(q(a)-q(a'))][v(a)\cdot(q(a)-q(a'))]\d l^T(a')\\
&
-
\frac{1}{\sigma}\ \int_{N}e_{a,a'}[v(a')\cdot v(a)]\d l^T(a').
\end{aligned}
$$
The normal Hamiltonian is given by
$$
\begin{aligned}
H^{\mathcal{S}}(q,v,p,l)&=\frac{1}{2}\iint_{N^2}e_{a,a'}[\d p(a)\cdot \d p(a')]
-
\frac{1}{2\sigma}\int_{N^2}e_{a,a'}[v(a')\cdot v(a)][\d l(a)\cdot \d l(a')]
\\
&
-
\frac{1}{2\sigma}\iint_{N^2}e_{a,a'}[(v(a))\cdot(q(a)-q(a'))][\d l(a)\cdot \d p(a')+\d p(a)\cdot \d l(a')]
\\
&
+
\frac{1}{2\sigma^2}
\iint_{N^2}e_{a,a'}[v(a')\cdot(q(a)-q(a'))][v(a)\cdot(q(a)-q(a'))][\d l(a)\cdot \d l(a')].
\end{aligned}
$$
From there, the geodesic equations are easily deduced. For example
$$
\begin{aligned}
\dot{l}(a)^T=-\partial_vH^{\mathcal{S}}(q,v,p,l)^T(a)=&
\frac{1}{\sigma}\int_{N}e_{a,a'}v(a')[l(a)\cdot \d l(a')]
\\
&
+\frac{1}{2\sigma}\int_{N}e_{a,a'}(q(a)-q(a'))[l(a)\cdot \d p(a')+p(a)\cdot \d l(a')]
\\
&
-
\frac{1}{\sigma^2}
\int_{N}e_{a,a'}(q(a)-q(a'))[v(a')\cdot(q(a)-q(a'))][l(a)\cdot \d l(a')].
\end{aligned}
$$

In the case of landmarks, we take $q=(x_1,\dots,x_n)$, $v=(v_1,\dots,v_n)$, $p=(p_1,\dots,p_n)$, and $l=(l_1,\dots,l_n)$. Denoting $x_{i,j}=x_i-x_j$,we get
$$
\left.
\begin{aligned}
\dot{x}_i^T=&\sum_{j=1}^n\left(p_j-\frac{1}{\sigma}[v_j\cdot x_{i,j}] l_j\right)e(x_{i,j}),\\
\dot{v}_i^T=&\sum_{j=1}^n\left(-\frac{1}{\sigma}[v_i\cdot x_{i,j}]p_j -\frac{1}{\sigma^2}[v_i\cdot(x_{i,j})][v_j\cdot x_{i,j}]l_j+\frac{1}{\sigma}[v_i\cdot v_j]l_j\right)e(x_{i,j}),
\\
\dot{l}_i^T=&\frac{1}{\sigma}\sum_{j=1}^n [l_i\cdot l_j]e(x_{i,j})v_j+\frac{1}{2\sigma}\sum_{j=1}^n [l_i\cdot p_j+l_ip_j]e(x_{i,j})x_{i,j}\\&-
\frac{1}{\sigma^2}\sum_{j=1}^n [v_j\cdot x_{i,j}][l_i\cdot l_j]
e(x_{i,j})x_{i,j}.
\end{aligned}\right\rbrace\qquad i=1,\dots,n
$$
The derivative of $p_i$ is slightly more complex, as
$$
\left.
\begin{aligned}
\dot{p}_i^T&=\frac{1}{\sigma}\sum_{j=1}^n\left(p_i\cdot p_j-\frac{1}{\sigma}[v_i\cdot x_{i,j}][v_j\cdot x_{i,j}][l_j\cdot l_j]-\frac{1}{\sigma}[v_i\cdot v_j][l_i\cdot l_j]\right)e(x_{i,j}) x_{i,j}
\\
&
-\frac{1}{\sigma^2}\sum_{j=1}^n\left([v_i\cdot x_{i,j}][l_i\cdot p_j]-[v_j\cdot x_{i,j}][p_i\cdot l_j]\right)e(x_{i,j}) x_{i,j}
\\
&
- \frac{1}{\sigma^3}\sum_{j=1}^n\left([v_i\cdot x_{i,j}][v_j\cdot x_{i,j}][l_i\cdot l_j]\right)e(x_{i,j})x_{i,j}
\\
&
+\frac{1}{\sigma}\sum_{j=1}^n\left([l_i\cdot p_j]v_i-[p_i\cdot l_j]v_j\right)e(x_{i,j}).
\end{aligned}\right\rbrace\qquad i=1,\dots,n.
$$

\begin{remark}
Again, in the case of landmarks, there is an open and dense subset of $\hat{\mathcal{S}}$ on which the sub-Riemannian structure is actually Riemannian.
\end{remark}

\subsubsection{The case of sub-Riemannian Gaussian kernels}

Here, we compute the sub-Riemannian Hamiltonian geodesic equations for the shape space $\mathcal{S}=\mathcal{C}^0(N,\R^d)$, with $N$ a compact Riemannian manifold, and with $V$ defined by the reproducing kernel
$$
K(x,y)p=e^{-\frac{\vert x-y\vert^2}{2\sigma}}\sum_{r=1}^k p(X_r(y)) X_r(x)=e(x-y)\sum_{r=1}^k p(X_r(y)) X_r(x),
$$
for $X_1,\dots,X_k$ smooth bounded vector fields on $\R^d$ with bounded derivatives at every order, linearly independent at every point $x$ of $\R^d$. Note that any $X\in V$ is everywhere tangent to the distribution of subspaces of dimension $k$ generated by the $X_r$s.

Now for $q\in\mathcal{S}$ and $p=(p_1,\dots,p_d)\in T^*_q\mathcal{S}=\mathcal{M}(N)\otimes(\R^d)^*$, we get
$$
(K^V\xi_q^*p)(x)=\sum_{r=1}^k\int_{N}e(x-q(a))\d p(a)(X_r(q(a)))X_r(x),
$$
where $\d p(a)(X_r(q(a)))\in \mathcal{M}(N)$ is defined by $X_r^1( q(a)) dp_1(a)+\cdots+X_r^d(q(a)) dp_d(a)$.

Then we get
$$
K_qp(a)=K^V\xi_q^*p(q(a))=\sum_{r=1}^k\int_{N}e(q(a)-q(a'))\d p(a')(X_r(q(a')))X_r(q(a)),
$$
and compute the reduced Hamiltonian 
$$
H^{\mathcal{S}}(q,p)=\frac{1}{2}\sum_{r=1}^k \iint_{N^2}e(q(a)-q(a'))\d p(a)(X_r(q(a))) \d p(a')(X_r(q(a'))).
$$
Then, reintroducing the notations $e_{a,a'}=e(q(a)-q(a'))$, we obtain
$$
\begin{aligned}
\partial_qH^\mathcal{S}(q,p).\delta q=-&
\frac{1}{\sigma}\sum_{r=1}^k \iint_{N^2}e_{a,a'}[\delta q(a)\cdot (q(a)-q(a'))]\d p(a)(X_r(q(a))) \d p(a')(X_r(q(a')))
\\
&
+\sum_{r=1}^k \iint_{N^2}e_{a,a'}\d p(a)(\d X_{r}(q(a)).\delta q(a)\, \d p(a')(X_r(q(a'))).
\end{aligned}
$$
In other words, as a Radon measure,
$$
\begin{aligned}
\partial_qH^\mathcal{S}(q,p)(a)=&\sum_{r=1}^k
\int_{N}e_{a,a'}\d p(a')(X_r(q(a')))\left(p(a)(\d X_{r}(q(a)) -\frac{1}{\sigma}p(a)(X_r(q(a)))(q(a)-q(a'))^T\right).
\end{aligned}
$$
Here, $p(\d X_{r}(q))$ in $\mathcal{M}(N)\otimes(\R^d)^*$ is the $(\R^d)^*$-valued Radon measure given by
$$
\d p(a)(\d X_{r}(q(a))=dp_1(a)\otimes \d X^1_{r}(q(a))+\cdots+\d p_d(a)\otimes \d X^d_{r}(q(a)).
$$
Therefore, the Hamiltonian geodesic equations are given by
$$
\begin{aligned}
\dot{q}(a)=&\sum_{r=1}^k\int_{N}e_{a,a'}\d p(a')(X_r(q(a')))X_r(q(a)),
\\
\dot{p}(a)=&\sum_{r=1}^k
\int_{N}e_{a,a'}\d p(a')(X_r(q(a')))\left(\frac{1}{\sigma}p(a)(X_r(q(a)))(q(a)-q(a'))^T-p(a)(\d X_{r}(q(a))) \right).
\end{aligned}
$$
In the case of landmarks, $q=(x_1,\dots,x_n)\in (\R^d)^n$ and $p=(p_1,\dots,p_n)\in (\R^d)^{n*}$, and denoting $x_i-x_j=x_{ij}$, we obtain
\begin{equation}
\label{geodsrlmk}
\begin{aligned}
\dot{x}_i=&\sum_{r=1}^k\sum_{j=1}^n e(x_{ij}) p_j(X_r(x_j))X_r(x_i),
\\
\dot{p}_i=&\sum_{r=1}^k
\sum_{j=1}^n e(x_{ij})p_j(X_r(x_j))\left(\frac{1}{\sigma}p_i(X_r(x_i)) x_{ij}^T
-p_i(\d X_{r}(x_i))\right).
\end{aligned}
\end{equation}

\begin{remark}
An interesting fact is that $e(x_{ij})\rightarrow \delta_{ij}$ as $\sigma\rightarrow 0$, where $\delta_{ij}=1$ if $i=j$ and $0$ otherwise. Hence $\frac{1}{\sigma}e(x_{ij})x_{ij}\rightarrow 0$ as $\sigma\rightarrow 0$ because $x_{ii}=0$. Therefore, as $\sigma$ goes to $0$, the geodesic equations become
$$
\begin{aligned}
\dot{x}_i=&\sum_{r=1}^kp_i(X_r(x_i))X_r(x_i),
\\
\dot{p}_i=&\sum_{r=1}^k
-p_i(X_r(x_i))p_i(\d X_{r}(x_i)).
\end{aligned}
$$
For each couple $t\mapsto(x_i(t),p_i(t))\in T^*\R^d$, we recognize the hamiltonian geodesic equation for the sub-Riemannian structure $\Delta$ induced on $\R^d$ by $X_1,\dots,X_k$ \cite{MBOOK}: the dynamic is that of a system of $n$ points without interaction in $(\R^d,\Delta)$.

So Equations \eqref{geodsrlmk} can be seen as a perturbation of such a dynamical by creating an interaction between the points $x_i$, which decreases exponentially with respect to the square of the Eucliddean distances between the points.
\end{remark}

\section{Equivariant mappings and lifted shape spaces}\label{sec_lifted}

The general purpose of this section is to define and study the applications of equivariant mappings between shape spaces. In particular, we will use them to study symmetries in both Hamiltonian flows, and to consider ILH shape spaces of infinite order. Another use for equivariant maps is the introduction of lifted shape spaces, which allow the consideration of additional information along deformations, and are always strictly sub-Riemannian even for finite dimensional shape spaces.


\subsection{Equivariant maps}\label{sec:equivmap}

We keep the notations from the previous sections.

\begin{definition}
A mapping $\pi:\hat \s:\rightarrow\s$ between two shape spaces of respective order $\hat \ell\geq\ell$ is said to be \textbf{equivariant} if for every diffeomorphism $\varphi$ and state $\hat{q}\in \hat \s$, $\phi\cdot \pi(\hat q)=\pi(\varphi\cdot\hat q)$. 
\end{definition}
The composition rule for derivatives shows that for every $X\in \Gamma^{s_0+\hat{\ell}}(TM)$, we have 
$$
d\pi(\hat q).\xi_{\hat{q}}X=\xi_{\pi(\hat{q})}X. 
$$
In particular, a horizontal curve on $\hat\s$ is projected to a horizontal curve on $\s$ with the control $X(\cdot)$. Let us give a few examples.

\begin{example} \textbf{Additional landmarks:} Take $n\leq\hat{n}\in\N^*$, and $\mathcal{S}=Lmk_n(M)$ and $\hat{\mathcal{S}}=Lmk_{\hat{n}}(M)$. We define $\pi:Lmk_{\hat{n}}(M)\rightarrow Lmk_n(M)$ by
$$
\pi(x_1,\dots,x_{\hat{n}})=(x_1,\dots,x_n).
$$
The mapping $\pi$ is an equivariant submersion for the diagonal action of $\mathcal{D}^{s_0+1}(M)$: for every $\varphi$ in $\mathcal{D}^{s_0+1}(M)$,
$$
\pi(\varphi\cdot (x_1,\dots,x_{\hat{n}}))=\pi(\varphi(x_1),\dots,\varphi(x_{\hat{n}}))=(\varphi(x_1),\dots,\varphi(x_n))=\varphi\cdot\pi(x_1,\dots,x_{{n}}).
$$
\end{example}
This is the mapping used in \cite{Y1} for the diffeon method.

\begin{example} \textbf{Embeddings of submanifolds:} Let $\hat{N}$ be a smooth compact Riemannian manifold, $N\subset\hat N$ a submanifold of $\hat{N}$ without boundary, and let $\mathcal{S}=\mathcal{C}^0(S,M)$ and $\hat{\mathcal{S}}=\mathcal{C}^0(\hat{S},M)$. Then the map $\pi:\hat{q}\mapsto \hat{q}_{\vert N}$ is an equivariant submersion.
\end{example}
This was used in \cite{ATY} to prove that the LDDMM methods on discrete shapes converge to the continuous case as the discretization gets finer and finer.

\begin{example} \textbf{Action on the tangent space:} Let $\mathcal{S}$ be a shape space in $M$ of order $\ell$. Then $\hat{\mathcal{S}}=T\mathcal{S}$ is a shape space of order $\ell+1$ for the action of $\mathcal{D}^{s_0+\ell+1}(M)$ given by the differential of the action on $\mathcal{S}$
$$
\varphi\cdot(q,v)=(\varphi\cdot q,\partial_q(\varphi\cdot q)(v))\in T_{\varphi\cdot q}\mathcal{S}.
$$
Then the projection $\pi(q,v)=q$ is an equivariant submersion for these actions.

In the case of landmarks, i.e. $\mathcal{S}=Lmk_n(M)$, we can write $(q,v)=(x_1,\dots,x_n,v_1,\dots,v_n)$, with $v_i\in T_{x_i}M$, and
$$
\varphi\cdot(q,v)=(\varphi(x_1),\dots,\varphi(x_n),d\varphi_{x_1}(v_1),\dots,d\varphi_{x_n}(v_n)).
$$
\end{example}

\begin{example} \textbf{Inclusion in a shape space of lower order:} Consider $\s=\mathrm{Emb}^{\alpha_0+\ell}(N,M)$ and $\hat\s=\mathrm{Emb}^{\alpha_0+\ell+1}(N,M)$, with $N$ a smooth compact manifold. Then the inclusion $\hat\s\hookrightarrow\s$ is trivially equivariant.
\end{example}
Depending on the properties of $\pi$, many things can be deduced from the existence of such a mapping.

\subsection{Equivariant mappings and the geodesic flow}\label{sec:equivgeod}

We keep the notations and setting of the previous section, and take $\pi$ of class $\mathcal{C}^\infty$.
Fix $V$ a Hilbert space of vector fields with continuous  inclusion in $\Gamma^{s_0+\hat\ell+k}(TM)$ for $k\geq1$. Then a horizontal geodesic on $\s$ can be pulled back to a geodesic on $\hat\s$ through $\pi$.

\begin{proposition}
Let $(q(\cdot),p(\cdot))$ be a normal geodesic (resp. a singular curve) and the corresponding covector on $T^*\s$, such that $q(0)=\pi(\hat{q}_0)$. Let $X(\cdot)$ be the corresponding vector field in $L^2(0,1;V)$. Then for $\hat{q}(\cdot)=\varphi^X(\cdot)\cdot\hat{q}_0$ and $\hat p(\cdot))=d\pi(\hat q(\cdot)^*p(\cdot)$, $(\hat q(\cdot),\hat p(\cdot))$ is a normal geodesic (resp. a singular curve).

Conversely, let $(\hat q(\cdot),\hat p(\cdot))$ be a normal geodesic (resp. a singular curve)  on $\hat{\s}$, and denote $q(\cdot)=\pi(\hat{q}(\cdot))$. If $\hat{p}(0)=d\pi(\hat{q}(0))^*p_0$ for some $p_0\in T^*_{\hat q_0}\s$, then $q(\cdot)$ is a geodesic (resp. an abnormal curve) with initial covector $p_0$.
\end{proposition}
\begin{proof}
Let $(q(\cdot),p(\cdot))$ be a normal geodesic (resp. a singular curve) and the corresponding covector on $T^*\s$, such that $q(0)\in \pi(\hat{q}_0)$. Define $(\hat{q}(\cdot),\hat{p}(\cdot))$ as in the hypothesis. We just need to prove that $\mu^{\hat \s}(\hat q(t),\hat p(t))=\mu^{ \s}(q(t),p(t))$ thanks to Propositions \ref{prop:mom1} and \ref{prop:mom2}. But since $d\pi({\hat{q}}).\xi_{\hat{q}}=\xi_{\pi(\hat q)}$ we have $\xi_{\pi(\hat q)}^*=d\pi(\hat q)^*\xi_{\hat q}^*$, so for all $t\in[0,1]$, we do get
$$
\mu^\s(q(t),p(t))=\xi_{q(t)}^*p(t)=\xi_{\hat q(t)}^*\hat p(t)=\mu^{\hat\s}(\hat{q}(t),\hat{p}(t)).
$$

Conversely,  let $(\hat q(\cdot),\hat p(\cdot))$ be a normal geodesic (resp. a singular curve)  on $\hat{\s}$ such that $\hat{p}(0)=d\pi(\hat{q}(0))^*p_0$ for some $p_0\in T^*_{\hat q_0}\s$, and denote $q(\cdot)=\pi(\hat{q}(\cdot))$. Then the flow of the control $X(\cdot)$ corresponding to $\hat q(\cdot)$ is a geodesic (resp. singular curve) from $e$ on $\D^s(M)$ with initial momentum $\xi_{\hat q(0)}^*\hat{p}(0)=\xi_{q(0)}^*p_0,$ and therefore induces a normal geodesic (resp. singular curve) on $\s$ with initial covector $p_0$. But since for every $t$ in $[0,1]$
$$
\varphi^X(t)\cdot q(0)=\pi(\varphi^X(t)\cdot \hat{q}(0))=
\pi(\hat{q}(t))=q(t),
$$
we proved the converse.
\end{proof}
This result can be used to find symmetries in the geodesic flow of $\hat\s$, for example when $\s$ is the quotient of $\hat\s$ with a group action that commutes with that of $\D^{s_0+\hat\ell}(M)$. Indeed, it shows that the covector $\hat{p}(\cdot)$ stays of the form $d\pi(q(t))^*p(t)$ for all $t$ as long is had this form at $t=0$.

When $\pi$ is an immersion, can also use it to find elusive geodesics on $\s$, and also consider ILH shape spaces of infinite order.

\subsubsection{ILH shape spaces of infinite order} 

An ILH shape space of infinite order is an Inverse Limit Hilbert space $\s^\infty$ (see \cite{KMBOOK,O}), inverse limit of a decreasing sequence $\s^{i+1}\subset\s^i$, with each $\s^i$ a shape space of increasing order $\ell_i>\ell_{i-1}$, such that the inclusions $\s^{i+1}\hookrightarrow\s^{i}$ are continuous, equivariant and have dense image. As a consequence, $\D^\infty(M)$ acts smoothly on $\s^\infty$. 
\begin{example}
A typical example is $\s^\infty=\mathrm{Emb}^\infty(S,M)$, the manifold of all smooth embeddings of a compact manifold $S$ in $M$. It is the inverse limit of the sequence of Hilbert manifolds $\mathrm{Emb}^{\alpha_0+i}(S,M),$ $i\in\N$ with $\alpha_0>\dim(S)/2$.
\end{example}
When $V$ is a Hilbert space of vector fields with continuous inclusion in every $\Gamma^{s}(TM)$, it defines a sub-Riemannian structure of class $\mathcal{C}^\infty$ on each $\s^i$, and also on $\s^\infty$, through the infinitesimal action.

There is no problem with defining the normal and abnormal Hamiltonians on $\s^\infty$ as usual, but the proofs of Section 2 do not work: linear differential equations may not have a unique solution on ILH spaces. However, it is very easy to check that, for fixed $q_0\in\s^\infty$ and $p\in T_{q_0}^*\s^i$, we also have $p\in T_{q_0}^*\s^{i+j}$ for every $j\geq 0$ and
$$
H^{\s^i}(q,p)=H^{\s^{i+j}}(q,p),\quad \nabla^\omega H^{\s^i}(q,p)=\nabla^\omega H^{\s^{i+j}}(q,p).
$$
Consequently, $H^{\s^\infty}(q,p)=H^{\s^i}(q,p)$ and $H^{\s^\infty}(q,p)=H^{\s^i}(q,p)$. Therefore, $(q(\cdot),p(\cdot))$ is a solution of the normal geodesic equation in $T^*\s^\infty$ if and only if it is a solution to the normal geodesic equation in $T^*\s^i$, and we know that this last equation has a unique solution defined on all of $\R$. Since $T^*_{q_0}\s^\infty=\cup_{i\geq0}T_{q_0}^*\s^i$, we obtain the existence of a global geodesic flow on $T^*\s^\infty$.

This discussion shows that, on an ILH shape space of infinite order, the Hamiltonian geodesic flow is still complete, and we can solve the geodesic equations as long as we can solve them in each $\s^i$.

\subsubsection{Finding symmetries in the normal geodesic equations}

In this section, we give a method to investigate symmetries on a shape space $\s$ on which a group $G$ acts on the right with an equivariant action with respect to that of $\D^s(M)$. The goal is to prove that the momentum for the action of $G$ is constant along normal geodesics and singular curves, in a way that is similar to what happens with classical Clebsch optimal control problems with symmetries \cite{CH,MRBOOK,RV}. The problem is that in our setting, we need to consider actions that do not come from Lie groups: $G$ will usually itself be a group of diffeomorphisms of fixed regularity.  We will solve this problem in the most important example in a way that can easily be generalized to a wider range of problems.

We will only consider shape spaces of the form $\s= \mathrm{Emb}^{\alpha_0+\ell}(N,M)$, with $\ell\geq1$, $N$ a smooth compact manifold, and $\alpha_0$ the smallest integer greater than $\dim(N)/2$. The action of $\D^s(M)$ is given by the composition on the left: $\varphi\cdot q=\varphi\circ q$, so that $\xi_qX=X\circ q$.  We want to consider the symmetry group $G=\mathcal{D}^{\alpha_0+\ell}(N)$, which acts on $\s$ contnuously on the right by right composition: $q\cdot\eta=q\circ\eta$ for $q$ in $\s$ and $\eta$ in $G$. This action is indeed equivariant: $\varphi\circ(q\circ\eta)=(\varphi\circ q)\circ\eta$. Note that $G$ is a Hilbert manifold but only a continuous group, and its action $\s$ is only continuous.

However, if $q\in \mathrm{Emb}^{\alpha_0+\ell+1}(N,M)$, the omega lemma (see \cite{EM,O,S}) implies that $\eta\mapsto q\circ\eta$ is of class $\mathcal{C}^1$. Consequently,  we have an infinitesimal action 
$$
\zeta:T_eG=\Gamma^{\alpha_0+\ell}(TN)\rightarrow T_{q}\s
$$
given by $\zeta_{q}(Y)(a)=\d q(a).Y(a)$ for every $a\in S$. This allows us to define a momentum map \emph{at} $q$ for the action of $G$ by
$$
\mathfrak{m}(q,p)=\zeta_q^*p=\d q^*p\in T^*_eG=\Gamma^{-\alpha_0-\ell}(T^*N)
$$
Fix $V$ a Hilbert space of vector fields on $M$ that has continuous inclusion in $\Gamma^{s_0+\ell+k}(TM)$ for $k\geq1$. Also fix an initial point $q_0\in\s$, and assume that $q_0\in \mathrm{Emb}^{\alpha_0+\ell+1}(N,M)$. Let $(q(\cdot),p(\cdot))$ be a horizontal geodesic (resp. an abnormal curve) for the sub-Riemannian structure induced by $V$ on $\s$, and $X(\cdot)$ the corresponding control. Then, for every time $t$, $q(t)=\varphi^X(t)\circ q_0$. But since $X\in L^2(0,1;V)$, $\varphi^X(t)$ is of class $H^{s_0+\ell+1}$, so $q(t)\in \mathrm{Emb}^{\alpha_0+\ell+1}(N,M)$. Therefore, the momentum $\mathfrak{m}(t)=\mathfrak{m}(q(t),p(t))$ with respect to the action of $G$ is well-defined along the trajectory.
\begin{proposition}  \label{propo:preservedmom}
Let $(q(\cdot),p(\cdot))$ be a horizontal geodesic or an abnormal curve for the sub-Riemannian structure induced by $V$ on $\s$ with covector $p(\cdot)$, and $X(\cdot)$ the corresponding control. Assume that $q(0)$ belongs to the dense subset $\mathrm{Emb}^{\alpha_0+\ell+1}(N,M)$. Then the momentum $\mathfrak{m}(t)=\mathfrak{m}(q(t),p(t))$ with respect to the action of $G$ is well-defined and \textbf{constant} along the trajectory: $\mathfrak{m}(t)=\mathfrak{m}(0).$ 

In particular, if $k\geq2$, this momentum is constant under the Hamiltonian geodesic flow of $T^*\s$ as long as the starting point $q_0$ belongs to the dense subset $\mathrm{Emb}^{\alpha_0+\ell+1}(N,M)$.
\end{proposition}
\begin{proof}
Fix $Y\in \Gamma^{\alpha_0+\ell}(TN)=T_eG$. Then
$$
\mathfrak{m}(t)(Y)=p(t)(\d q(t).Y).
$$
Then $\mathfrak{m}(\cdot)(Y):[0,1]\rightarrow\R$ is of class $H^1$, because $p(\cdot)$ and $dq(\cdot)$ are of class $H^1$. This is trivial for $p(\cdot)$, and true for $dq(\cdot)$ because $(q(\cdot),dq(\cdot))$ is the deformation  induced by $X(\cdot)$ in the shape space $T\s$ of order $\ell+1$ of $(q_0,dq_0)$. Moreover, we have $\partial_t{\d q}=dX(t)(q(t)).dq(t)$. Therefore, for almost every $t$ in $[0,1]$,
$$
\partial_t\mathfrak{m}(t)(Y)=\dot{p}(t)(dq(t).Y)+p(t)(dX(t)(q(t)).dq(t).Y).
$$
But in canonical coordinates, since $\xi_{q(t)}X(t)=X(t)\circ q(t),$
$$
\dot{p}(t)(dq(t).Y)=-\partial_q(X(t)\circ q(t))^*p(t)(dq(t).Y)=-p(t)(\partial_q(X(t)\circ q(t)).dq(t).Y)=-p(t)(dX(t)(q(t)).dq(t).Y)
$$
and we do get $\partial_t(\mathfrak{m}(t)(Y))=0$.
\end{proof}

\begin{remark}
Note that when considering $\mathcal{S}=\mathrm{Emb}^\infty(N,M)$, the action on the right by $\D^\infty(N)$ is smooth and has a well-defined momentum map. Then the results of the previous section immediately combine with Proposition \ref{propo:preservedmom} to prove that this momentum is preserved under the Hamiltonian geodesic flow.
\end{remark}

\begin{remark}
The general case for abstract shape spaces would be as follows: assume a topological group $G$ with a Banach manifold structure acts continuously and equivariantly on the right on a shape space $\s$. Assume that $q_0$ is such that $\eta\in G\mapsto q_0\cdot\eta$ is of class $\mathcal{C}^1$, so that we can define a momentum map $\mathfrak{m}$ at $q_0$. Then, along any normal geodesic or singular curve $q(\cdot)$ from $q_0$ with covector $p(\cdot)$, the momentum $\mathfrak{m}(q(t),p(t))$ is well-defined and constant. 
\end{remark}

\paragraph{Application: inexact matching between shapes up to reparametrization.} The main application of Proposition \ref{propo:preservedmom} is to study problems that do not depend on a reparametrization $q\circ\eta$ of the state $q$. For example, assume we want to minimize
$$
J(X(\cdot))=A(X(\cdot))+g(q(1))
$$
over all horizontal systems $(q,X)$, where $g:\s\rightarrow\R$ is of class $\mathcal{C}^1$ and $g(q)$ only depends on Im$(q)$ (for example, $g$ could be a norm on currents \cite{GTY1,CTr}). This is equivalent to saying that $g(q\cdot \eta)=g(q)$ for every reparametrization $\eta$: if Im$(q)$=Im$(q')$, then $q=q'\circ (q'^{-1}\circ q)$. We know that the trajectory $q(\cdot)$ of every critical point of $J$ is a normal geodesic whose covector $p(\cdot)$ satisfies $p(1)=-dg(q(1)).$ 

Now, if $q_0$ is in $\mathrm{Emb}^{\alpha_0+\ell+1}(N,M)$, then $\mathfrak{m}(q(1),p(1))=\mathfrak{m}(q(1),-dg(q(1)))$. But since $g$ is invariant under the action of $G$, we have $\mathfrak{m}(q(1),-dg(q(1)))=0$. Therefore, $\mathfrak{m}(q(t),p(t))=0$ for every $t$, since the momentum is conserved along a normal geodesic.

Hence, when looking for minimizers of $J$, it is enough to consider the control system
$$
\dot{q}(t)=K_{q(t)}u(t),
$$ 
where $u(t)\in T_{q(t)}^*\s$ and $\mathfrak{m}(q(t),u(t))=0$ for almost every $t$. The cost becomes
$$
J(q(\cdot),u(\cdot))=\frac{1}{2}\int_0^1u(t)(K_{q(t)}u(t))dt+g(q(1)).
$$
Morover, $\mathfrak{m}(q(t),u(t))=0$ if and only if $u(t)$, identified with a vector field (with distributional coefficients) along Im$(q(t))$ thanks to the Riemannian metric on $M$, is everywhere orthogonal to the submanifold Im$(q(t))$. We recover the well-known (but never proved in this general setting) and widely used fact that,  when matching submanifolds, the covector must be orthogonal to the surface \cite{BHM,MM} along a minimizing deformation.

\begin{remark}
Exact matching of submanifold is much more complex, because we may again see the appearance of abnormal and elusive geodesics. However, we can still obtain a global Hamiltonian geodesic flow from any initial point $q_0\in \mathrm{Emb}^{\alpha_0+\ell+1}(N,M)$, simply by taking an initial covector $p_0$ with $\mathfrak{m}(q_0,p_0)=0$. This will yield curves that are critical points of the action among all horizontal curve with endpoints in the same $G$-orbit.
\end{remark}

\subsection{Equivariant submersions and lifted structures}

The structures defined in Section 2 are usually not Riemannian for an infinite dimensional shape space $\s$. For example, for $\s=\mathrm{Emb}^{\alpha_0+\ell}(N,M)$, we require vector fields that are more regular than the embeddings  we are studying so we cannot obtain the whole tangent bundle of $\s$. However, in most practical applications, the shape spaces are discretized into landmark spaces, for which the sub-Riemannian structure induced by a space $V$ of vector fields is, itself, Riemannian.

Sub-Riemannian geometry still has many practical uses in shape analysis. In \cite{ATY}, shape spaces with constraints that are linear with respect to the control were studied, which is precisely the framework of sub-Riemannian geometry. On the other hand, in \cite{Y1}, Younes gave a new kind of LDDMM methods using what he called diffeons, in which the momenta used to induce the vector fields deforming the shape have support outside of the shape, but those supports are still transported by the deformation. This leads to a sub-Riemannian setting as well. In this section, we will establish a framework encompassing the diffeon method in a broader setting using equivariant submersions.

The main use of this setting is to keep track of additional information on a given shape space, or to emphasize the importance of certain points in the shape. The resulting normal geodesic equations are more complex. We give some examples at the end of the section.

Let us start with a simple case to clarify the concepts we wish to introduce. Let $\mathcal{S}$ and $\mathcal{S}'$ be two shape spaces of respective orders $\ell$ and $\ell'$. Then $\hat{\mathcal{S}}=\mathcal{S}\times\mathcal{S}'$ is a shape space of order $\hat{\ell}=\max(\ell,\ell')$ for the diagonal action of $\mathcal{D}^{s_0+\hat{\ell}}(M)$
$$
\varphi\cdot (q,q')=(\varphi\cdot q,\varphi\cdot q').
$$ 
Fix $k\geq1$ and a Hilbert space of vector fields $V$ with continuous inclusion in $\Gamma^{s_0+\hat{\ell}+k}(TM)$. $V$ induces on each of $\mathcal{S}$, $\mathcal{S}'$ and $\hat{\mathcal{S}}$ a sub-Riemannian structure $SR_V^\s$, $SR_V^{\s'}$, $SR_V^{\hat\s}$, as described in Section \ref{shsp}. We will denote by $\xi$ all infinitesimal actions on these three spaces, the index making it clear which action is being used. We can then study simultaneous deformations of $q_0$ in $\s$ and $q_0'$ in $\s'$ by considering the product space $\hat\s$.

A different interesting problem to consider is the following. Fix $(q_0,q_0')\in \s\times\s'$, and consider a horizontal deformation $q(\cdot)$ of $q_0$ in $\s$. What kind of deformation of $q_0'$ does $q(\cdot)$ induce? 

More precisely, $q(\cdot)$ can be lifted to a unique minimal curve $\varphi^X$ in $\D^s(M)$ with minimal action. This lift is the flow of the vector field $X(t)=\xi_{q(t)}^{-1}\dot{q}(t)$, where $\xi_{q(t)}^{-1}:\xi_{q(t)}(V)\rightarrow (\ker\xi_{q(t)})^\perp$ is the pseudo-inverse of $\xi_{q(t)}$. It induces, in turn, a horizontal deformation
$$
{q}'(t)=({q}(t),q'(t))=(\varphi(t)\circ q_0,\varphi(t)\circ q'_0)
$$ 
of $q_0'$ in ${\mathcal{S}'}$, with
$$
\dot{{q}}'(t)=\xi_{{q'}(t)}\xi_{q(t)}^{-1}\dot{q}(t).
$$

This yields a horizontal curve $\hat{q}(\cdot)=(q(\cdot),{q}'(\cdot))$ in $\hat\s$ with
$$
\dot{\hat{q}}(t)=\xi_{\hat q(t)}\xi_{q(t)}^{-1}\dot{q}(t).
$$
\begin{remark}
The curve $\hat q(\cdot)$ is, in fact, the only one from $\hat q_0$ whose action in the sense of Remark \ref{rk:minlift} coincides with that of $q(\cdot)$. In other words, it is the unique minimizing lift of $q(\cdot)$ to $\hat{\mathcal{S}}$ through $\pi$ from $\hat{q}_0$. The converse is trivially true: if $\pi(\hat q(\cdot))=q(\cdot)$ and their action in the sense of Remark \ref{rk:minlift} coincide, then $\dot{\hat q}(t)=\xi_{\hat q(t)}\xi_{q(t)}^{-1}\dot{q}(t)$ almost everywhere.
\end{remark}

We obtain a new sub-Riemannian structure $\mathrm{SR}_{V,\s}^{\hat \s}=(\mathcal E,g^\s,\chi)$, with 
$$\begin{aligned}
\mathcal E_{(q,q')}&=\xi_q(V),\\
g^\s_{(q,q')}(u,v)&=\langle\xi_q^{-1}u,\xi_q^{-1}v\rangle,
\end{aligned}
$$ 
and 
$$
\chi_{(q,q')}u=\xi_{(q,q')}\xi_q^{-1}u.
$$

\begin{remark} 
This structure is particularly hard to study in general, because $\xi_q^-1$ is very hard to compute, and may not be even continuous in $q$. This is why our results will assume that $\xi_q(V)$ is closed in $T_q\s$ for every $q$, which will mean that $\xi_q(V)=K_q(T^*_q\s)$.
\end{remark}

This is a particular case of a lifted sub-riemannian structures on shape spaces through an equivariant submersion. We will see the general case in the next section.

We keep the notations and setting of the previous section. Going back to the end of Section \ref{sec:equivmap}, we could define a lifted sub-Riemannian structure on $\hat\s$ by lifting that of $\s$ through $\pi$. The vector bundle part would be
$$
\mathcal{E}_{\hat q}=\xi_{\pi(\hat q)}(V),
$$
the vector bundle morphism at $q$ would be given by $\xi_{\hat q}\xi_{\pi(\hat q)}^{-1}$, and the squared norm of $u\in \mathcal{E}_{\hat q}$ would be $\langle \xi_u,{\pi(\hat q)}^{-1}u\rangle$. However, this structure is not easy to study since $\xi_q^{-1}$ is hard to compute and may not even be continuous in $q$.

However, in the particular case where $\xi_q(V)$ is closed, Lemma \ref{lem:Kq} implies that $\xi_q(V)=$Im$(K_q)$, where $K_q=\xi_qK_V\xi_q^*$. This allows us to define the following lifted structure.

\begin{definition}\label{def:liftedsr}\textbf{Lifted sub-Riemannian structure:}
Let $\mathcal{S}$ and $\hat{\mathcal{S}}$ be shape spaces in $M$ of respective orders $\ell\leq\hat{\ell}$. Fix $k\geq1$, $V$ a Hilbert space of vector fields with continuous inclusion in $\Gamma^{s_0+\hat\ell}(TM)$, and $\pi:\hat\s\rightarrow\s$ an equivariant map of class $\mathcal{C}^\infty$. We assume that for every $q\in\s$, $\xi_q(V)$ is closed in $T_q\s$.

The \textbf{lifted sub-Riemannian structure} on $\hat{\s}$ is $\mathrm{SR}_\pi=(\pi^*T\s,g^K,\chi),$ where:
$$\left\lbrace
\begin{aligned}
(\pi^*T\s)_{\hat{q}}&=T^*_{\pi(\hat q)}\s\\
g_{\hat q}^K(u_1,u_2)&=u_1(K_{\pi(\hat q)}u_2),& u_1,u_2\in T^*_{\pi(\hat q)}\s,\\
\chi_{\hat q}&=\xi_{\hat{q}}K_V\xi_{\pi(\hat q)}^*.
\end{aligned}\right.
$$
We denote $A_\pi$ the corresponding action:
$$
A_\pi(\hat{q}(\cdot),u(\cdot))=\frac{1}{2}\int_0^1u(t)K_{\pi(\hat q(t))}u(t))dt.
$$
\end{definition}
\begin{remark}
The assumption that $\xi_q(V)$ be closed usually (although not always) implies that $\s$ is finite dimensional, for example a space of landmarks.
\end{remark}
\begin{remark}
We denote elements of $T^*\s$ by the letter $u$ instead of $p$ in order to stress that $u(\cdot)$ does not satisfy any of the Hamiltonian equations given in the previous section. The covector $u(\cdot)$ is simply a control, that can take any value in $T^*\s$.
\end{remark}
This structure is equivalent to the one we were looking for, as shown in the following proposition.
\begin{proposition}\label{hlift2}
A curve $\hat q(\cdot)\in H^1(0,1;\hat \s)$ is a horizontal for $\mathrm{SR}_\pi$ if and only if $\hat{q}(\cdot)=\varphi^X(\cdot)\cdot \pi(\hat q(0))$, with $X(t)=\xi_{\pi(\hat{q}(t))}^{-1}\dot{q}(t)$. Moreover, its action is equal to that of the minimal lift to $\D^{s_0+\ell}(M)$ of $\pi(\hat q(\cdot))$.
\end{proposition}
\begin{proof}
Denote $q(\cdot)=\pi(\hat q(\cdot))$, $\hat q_0=\hat q(0)$ and $q_0=q(0).$ Now, assume $\hat q(\cdot)$ is horizontal for $\mathrm{SR}_\pi$, so that, there exists $u:t\mapsto T_{q(t)}^*\s$ in $L^2$ such that for almost every $t\in [0,1]$,
$$
\dot{\hat q}(t)=\xi_{\hat q(t)}K_V\xi_{q(t)}^*u(t).
$$
But then
$$
\dot{q}(t)=d\pi(\hat q(t)).\xi_{\hat q(t)}K_V\xi_{q(t)}^*u(t).
$$
Since $d\pi(\hat q(t)).\xi_{\hat q(t)}=\xi_{\pi(\hat q)}$, we get
$$
\dot{q}(t)=\xi_{q(t)}K_V\xi_{q(t)}^*u(t)=K_{q(t)}u(t).
$$
But then Lemma \ref{lem:Kq} states that if $X(t)=K_V\xi_{q(t)}^*u(t)$, then $X(t)=\xi_{q(t)}^{-1}\dot{q}(t)$, so we do get that  $\dot{\hat{q}}(t)=\xi_{\hat q}\xi_{q(t)}^{-1}\dot{q}(t)$ and $\hat{q}(t)=\varphi(t)\cdot\hat{q}_0.$ Moreover, using the discussion from Section \ref{sec:minlift}, we get
$$
A(q(\cdot),X(\cdot))=\frac{1}{2}\int_0^1u(t)(K_{q(t)}u(t))dt=A_\pi(\hat{q}(\cdot),u(\cdot)).
$$
For the converse, assume that for $X(\cdot)=\xi_{q(\cdot)}^{-1}\dot{q}(\cdot)$, we have $\varphi^X(\cdot)\cdot\hat q_0$. Since  for every $t$ in $[0,1]$, by hypothesis, $\xi_{q(t)}(V)$ is closed, we know that there exists $u(t)\in T_{q(t)}^*\s$ such that $X(t)=K_V\xi_{q(t)}^*u(t)$. Moreover, 
$$
\int_0^1u(t)(K_{q(t)}u(t))dt=\int_0^1\langle X(t),X(t)\rangle dt.
$$
We just need to prove that $\dot{\hat q}(t)=\xi_{\hat q(t)}K_V\xi_{q(t)}^*u(t)$ for almost every $t$. But
$$
\xi_{\hat q(t)}K_V\xi_{q(t)}^*u(t)=\xi_{\hat q(t)}X(t)=\dot{\hat q}(t).
$$
\end{proof}

Let us give a few examples.

\begin{example} Take $n\leq\hat{n}\in\N^*$, and $\mathcal{S}=Lmk_n(M)$ and $\hat{\mathcal{S}}=Lmk_{\hat{n}}(M)$. We define $\pi:Lmk_{\hat{n}}(M)\rightarrow Lmk_n(M)$, as above, by $\pi(x_1,\dots,x_{\hat{n}})=(x_1,\dots,x_n)$. Now let $V$ be a hilbert subspace of vector fields of class at least $H^{s_0+1}$, with associated kernel $K(x,y):T_y^*M\rightarrow T_xM$.

Let $\hat{q}=(x_1,\dots,x_{\hat{n}})$ and $q=\pi(\hat{q})$. Now take some 
$$
p=(p_1,\dots,p_n)\in T_q^*Lmk_n(M)=T_{x_1}^*M\times\dots\times T_{x_n}^*M.
$$
Then for $x\in M$,
$$
X(x)=K^V\xi_q^*p(x)=\sum_{i=1}^nK(x,x_i)p_i,
$$
so we get
$$
d\pi^{-1}K_qp=\left(\sum_{i=1}^nK(x_1,x_i)p_i,\dots,\sum_{i=1}^nK(x_{\hat{n}},x_i)p_i\right).
$$

On the other hand, we have $\hat{p}=d\pi^*p=(p_1,\dots,p_n,0,\dots,0)$, so we also get
$$
K_{\hat{q}}\hat{p}=\left(\sum_{i=1}^nK(x_1,x_i)p_i,\dots,\sum_{i=1}^nK(x_{\hat{n}},x_i)p_i\right).
$$
\end{example}

\begin{example} Take $n\in\N^*$, $\mathcal{S}=Lmk_n(M)$ and $\hat{\mathcal{S}}=TLmk_{\hat{n}}(M)$. We take $\pi:TLmk_{\hat{n}}(M)\rightarrow Lmk_n(M)$ the usual projection. Now let $V$ be a hilbert subspace of vector fields of class at least $H^{s_0+2}$, with associated kernel $K(x,y):T_y^*M\rightarrow T_xM$.

Let $\hat{q}=(q,v)=(x_1,\dots,x_{\hat{n}},v_1,\dots,v_n)$, where $q=(x_1,\dots,x_n)$ and $v_i\in T_{x_i}M$. Now take some 
$
p=(p_1,\dots,p_n)\in T_q^*Lmk_n(M).
$
Again, we have $\hat{p}=d\pi^*p=(p_1,\dots,p_n,0,\dots,0)\in T^*TLmk_n(M)$, so we obtain
$$
K_{\hat{q}}\hat{p}=(w_1,\dots,w_n,w'_1,\dots,w'_n),
$$
with
$$
w_i=\sum_{j=1}^nK(x_i,x_j)p_j\in T_{x_i}M
$$
and
$$
w'_i=\partial_{x_i}\left(\sum_{j=1}^nK(x_i,x_j)p_j\right)(v_i)\in T_{(x_i,v_i)}TM,
$$
which is well-defined as the derivative at $x_i$ of the map $K^V\xi_q^*p:M\rightarrow TM$ applied to $v_i$.
\end{example}

%

Lifted shapes can be used for studying how the deformation of a given shape induces the deformation of a "bigger" shape, i.e. one with more information. It can also be used to induce a different kind of deformations, such as a deformation through diffeons \cite{Y1}.

We can then define horizontal curves for this new structure, the action $A$ and length $L$ of such curves (which coincide with the action and length for $\mathcal{H}^{\hat{\mathcal{S}}}$), and the corresponding sub-Riemannian distance $d^{\hat{\mathcal{S}}}_\pi$ on $\hat{\mathcal{S}}$.

We have the following result, a straightforward consequence of Proposition \ref{hlift2}.

\begin{proposition}\label{liftedhc}
Let $\hat{q}:[0,1]\rightarrow \hat{\mathcal{S}}$ be a curve of class $H^1$, and define $q:t\mapsto \pi(\hat{q}(t))$ its projection onto $\mathcal{S}$. Then
\begin{enumerate}
\item If $q$ is a geodesic (resp. a minimizing cuvre) on $\s$, so is $\hat{q}$, with respect to both $\mathrm{SR}_V^{\hat\s}$ and $\mathrm{SR}_\pi$.
\item As a consequence, $\pi$ is a lipshitz-1 map: for every $\hat{q}_1,\hat{q}_2\in \hat{M}, \ d_{\pi}(\hat{q}_1,\hat{q}_2)\geq d_{SR}^{\mathcal{S}}(\pi(\hat{q}_1),\pi(\hat{q}_2)),$ where $d_{\pi}^{\hat \s}$ is the distance associated to  $\mathrm{SR}_\pi$.
\end{enumerate}
\end{proposition}

\subsection{The lifted Hamiltonian geodesic equations}\label{hamilteqlifted}

In this section, we compute the normal Hamiltonian for a lifted structure, which will give us the normal geodesic equations.  We keep the notations and setting of the previous section. Moreover, we make the following assumption: 

\medskip
\centerline{(A1): The structure $\mathrm{SR}_V^\s$ is Riemannian, that is, $\xi_q(V)=T_q\s$ for every $q$ in $\s$.}
\medskip
This assumption implies that $K_q$ is a linear isomorphism. In most practical applications, this restricts us to the case where $\s$ is a finite dimensional manifold such as a landmark space, but this was already a consequence of assuming $\xi_q(V)$ closed in $T_q\s$.

The Hamiltonian $H^\pi:T^*\hat\s\rightarrow\s$ is given by
$$
H^\pi(\hat q,\hat p)=\max_{u\in T_{\pi(\hat q)}^*\s} \left(\underset{=H^\pi_1(\hat q,\hat p,u)}{\underbrace{\hat{p}(\xi_{\hat{q}}K_V\xi_{\pi(\hat  q)}^*u)-\frac{1}{2}u(K_{\pi(\hat q)}u)}}\right).
$$
Since, for fixed $(\hat q,\hat p)\in T^*\hat\s$, the right-hand side is strictly convex in $u$ thanks to (A1), the maximum is reached if and only if $\partial_uH^\pi_1(\hat q,\hat p,u)=0$, that is, if and only if 
$$
\xi_{\pi(\hat{q})}K_V\xi_{\hat q}^*\hat{p}=K_{\pi(\hat q)}u.
$$
But, because of Assumption (A1), this is equivalent to 
$$
u=u(\hat p,\hat q)=K_{\pi(\hat{q})}^{-1}\xi_{\pi(\hat{q})}K_V\xi_{\hat q}^*\hat{p}.
$$
We obtain for the normal Hamiltonian
$$
H^\pi(\hat{q},\hat{p})=H^\pi_1(\hat p,\hat q,u(\hat p,\hat q))=\frac{1}{2}u(\hat p,\hat q)(K_{\pi(q)}u(\hat p,\hat q)).
$$
Now computing the symplectic gradient of $H^\pi$ seems complicated, but when we use the fact that $\partial_uH^\pi_1(\hat p,\hat q,u(\hat p,\hat q))=0$ it gets much simpler. Considering that $K_{\pi( \hat q)}=d\pi(\hat q)K_{\hat q}d\pi(\hat q)^*$\footnote{Recall that $K_{\pi( \hat q)}=\xi_{\hat{q}}K_{\hat{q}}\xi_{\hat{q}}^*$}, we get
$$\left\lbrace
\begin{aligned}
\partial_pH^\pi(\hat{q},\hat{p})&=\xi_{\hat{q}}K_V\xi_{\pi(\hat  q)}^*u(\hat{q},\hat{p}),\\
-\partial_qH^\pi(\hat{q},\hat{p})&=
-(\partial_{\hat{q}}K_{\hat{q}}d\pi(\hat q)^*u)^*\left[\hat p-\frac{1}{2}d\pi(\hat{q})^*u \right]+\partial_{\hat{q}}[d\pi(\hat{q})^*u]^*K_{\hat{q}}^*\hat{p}.
\end{aligned}\right.
$$
Note that we took $u=u(\hat{q},\hat{p})$ for short and that none of its derivative are involved. This symplectic gradient is of class $\mathcal{C}^{k-1}$ and therefore, if $k\geq2$, possesses a well-defined flow. 

\begin{remark}
We did not prove that integral curves of this symplectic gradient project to sub-Riemannian geodesics. This is obviously true is $\hat{\s}$ is finite dimensional, since this is a well-known result in sub-Riemannian geometry. The case where $\hat{\s}$ is infinite dimensional is harder, because there is no natural differential structure on the space of horizontal systems if $\hat{\s}$ does not possess a local addition \cite{KMBOOK}. 
\end{remark}
\subsection{Some examples of lifted geodesic equations}\label{examplelifted}

We consider some simple cases of lifted shape spaces and compute their geodesic equations. Note that, contrarily to examples from Section \ref{ex2}, even the landmark cases will be strictly sub-Riemannian.

\paragraph{Extra landmarks with the Gaussian kernel}

Let $n,m\in \N^*$, and consider $\mathcal{S}=Lmk_n(\R^d)$ and $\hat{S}=Lmk_{n+m}(\R^d)$, with $\pi(x_1,\dots,x_{n+m})=(x_1,\dots,x_n)$. Then $T^*\mathcal{S}=Lmk_n(\R^d)\times (\R^d)^{n*}$, and $T^*\hat{\mathcal{S}}=Lmk_{m+n}(\R^d)\times (\R^d)^{(m+n)*}$.

We assume that the reproducing kernel $K$ of $V$ is the diagonal Gaussian kernel, as introduced in Section \ref{ex2}. Our goal is to compute the geodesic equations for this particular example. In other words, we first need to solve the linear system 
$$
K_{\pi(\hat{q})}u=\xi_{\pi(\hat{q})}K^V\xi_{\hat{q}}^*p.
$$
Then, we shall compute $\partial_{\hat{p}}H^\pi_1(\hat{q},\hat{p},u)=\xi_{\hat{q}}K^V\xi_{\pi(\hat{q})}^*u$ and $\partial_{\hat{q}}H^\pi_1(\hat{q},\hat{p},u)=(\hat{p}-\frac{1}{2}d\pi_{\hat{q}}^*u)
((\partial_{\hat{q}}K_{\hat{q}})d\pi_{\hat{q}}^*u).$ Note that the term $(\hat{p}-d\pi_{\hat{q}}^*u)
(K_{\hat{q}}((\partial_{\hat{q}}d\pi_{\hat{q}}^*)u)$ does not appear in $\partial_{\hat{q}}H^\pi_1(\hat{q},\hat{p},u)$ because $\partial_{\hat{q}}(d\pi_{\hat{q}}^*)=0$.

Here, $K_q$ can be identified to a $n\times n$ block matrix, with blocks of size $d\times d$, and the $(i,j)-$th bloc is given by $e(x_i-x_j)I_d$, with $I_d$ the identity matrix. It is invertible (see\cite{TY1}), and there exists $S=(S(x_i,x_j))_{i,j=1,\dots,n}$, a matrix of size $n\times n$ such that $\sum_{k=1}^nS(x_i,x_k)e(x_k-x_j)=1$ if $i=j$ and $0$ if $i\neq j$.

Then, for $\hat{q}=(x_1,\dots,x_{n+m})$, $u=(u_1,\dots,u_n)$ and $\hat{p}=(p_1,\dots,p_{n+m})$, we know that we have
$$
K^V\xi_{\hat{q}}^*\hat{p}(x)=\sum_{j=1}^{n+m}e(x-x_j)p_j^T.
$$
Hence, if  $(v_1,\dots,v_n)=\xi_qK^V\xi_{\hat{q}}^*\hat{p}$, we obtain
$$
v_j=\sum_{k=1}^{n+m}e(x_j-x_k)p_k^T,\quad j=1,\dots,n.
$$
So the equation $K_{\pi(\hat{q})}u=\xi_{\pi(\hat{q})}K^V\xi_{\hat{q}}^*p$ becomes
$$
\sum_{k=1}^ne(x_j-x_k)u_k^T=\sum_{k=1}^{n+m}e(x_j-x_k)p_j^T,\quad i=1,\dots,n.
$$
This linear equation is solved in $u$ by
$$
u_i=\sum_{j=1}^n\sum_{k=1}^{n+m}S(x_i,x_j)e(x_j,x_k)p_k=p_i+\sum_{j=1}^n
\sum_{k=n+1}^{n+m}S(x_i,x_j)e(x_j,x_{k})p_{k},
$$
and therefore, if $\partial_{\hat{p}}H^\pi_1(\hat{q},\hat{p},u)=\xi_{\hat{q}}K^V\xi_q^*u=(\hat{v}_1,\dots,\hat{v}_{n+m})$, we have
$$
v_{a}=\sum_{i=1}^ne(x_a,x_i)u_i^T=\sum_{i=1}^ne(x_a,x_i)p_i^T+\sum_{i,j=1}^n
\sum_{k=n+1}^{n+m}S(x_i,x_j)e(x_j,x_{k})p_{k}^T.
$$

Noting that $\partial_{\hat{q}}d\pi_{\hat{q}}=0$, we must now compute 
$$
\partial_{\hat{q}}H^\pi_1(\hat{q},\hat{p},u)=(\hat{p}-\frac{1}{2}d\pi_{\hat{q}}^*u)
((\partial_{\hat{q}}K_{\hat{q}})d\pi_{\hat{q}}^*u)+(\hat{p}-d\pi_{\hat{q}}^*u)=
(\hat{p}-\frac{1}{2}d\pi_{\hat{q}}^*u)
((\partial_{\hat{q}}K_{\hat{q}})d\pi_{\hat{q}}^*u).
$$ 
If we denote $\hat{p}'=\hat{p}-\frac{1}{2}d\pi_{\hat{q}}^*u=(p'_1,\dots,p'_{n+m})$ and write $-\partial_{\hat{q}}H^\pi_1(\hat{q},\hat{p},u)^T=(\alpha_1,\dots,\alpha_{n+m})$, we have
$$
\begin{aligned}
\alpha_a
=&
\ \ \ \ \frac{1}{\sigma}\ \sum_{i=1}^n\ \,e(x_a-x_j)[p_a'\cdot u_i+u_a\cdot p_i'](x_a-x_j)^T
\\
&+
\frac{1}{\sigma}\sum_{i=n+1}^{n+m} e(x_a-x_i)[u_a\cdot p_i'](x_a-x_i)^T,& a=&1,\dots,n,
\\
\alpha_{a}
=&
\ \ \ \ \frac{1}{\sigma}\ \sum_{i=1}^n\ \,e(x_{a}-x_i)[p'_{a}\cdot u_i](x_{a}-x_i)^T, &a=&n+1,\dots,n+m.
\end{aligned}
$$
On the other hand,
$
p'_a
={p_a}-\frac{1}{2}u_a$, hence, we obtain
$$
\begin{aligned}
\alpha_a
=&
\ \ \ \ \frac{1}{\sigma}\ \sum_{i=1}^n\ \, e(x_a-x_i)[p_a\cdot u_i+u_a\cdot p_i-u_a\cdot u_i](x_a-x_i)^T
\\
&+
\frac{1}{\sigma}\sum_{i=n+1}^{n+m} e(x_a-x_i)[u_a\cdot p_i](x_a-x_i)^T,& a=&1,\dots,n,
\\
\alpha_{a}
=&
\ \ \ \ \frac{1}{\sigma}\ \sum_{i=1}^n\ \,e(x_{a}-x_i)[p_{a}\cdot u_i](x_{a}-x_i)^T, &a=&n+1,\dots,n+m.
\end{aligned}
$$
We can finally write the lifted geodesic equations on $\hat{\mathcal{S}}$
$$
\begin{aligned}
u_a=&\ \ \ \ p_a+\sum_{i=1}^n
\,\sum_{k=n+1}^{n+m}S(x_a,x_i)e(x_i-x_{k})p_{k},&a=&1,\dots,n,
\\
\dot{x}_a=&\ \ \ \ \sum_{i=1}^ne(x_a-x_i)u_i^T
&a=&1,\dots,n+m,
\\
\dot{p}_a=&
\ \ \ \ \frac{1}{\sigma}\ \sum_{i=1}^n\ \, e(x_a-x_i)[p_a\cdot u_i+u_a\cdot p_i-u_a\cdot u_i](x_a-x_i)^T
\\
&+
\frac{1}{\sigma}\sum_{i=n+1}^{n+m} e(x_a-x_i)[u_a\cdot p_i](x_a-x_i)^T,& a=&1,\dots,n,
\\
\dot{p}_a
=&
\ \ \ \ \frac{1}{\sigma}\ \sum_{i=1}^n\ \,e(x_{a}-x_i)[p_{a}\cdot u_i](x_{a}-x_i)^T, &a=&n+1,\dots,n+m.
\end{aligned}
$$

\paragraph{Tangent spaces of landmarks with the Gaussian kernel}

In this example, we keep our Gaussian kernel $K$ as above and $\mathcal{S}=Lmk_n(\R^d)$ but we take $\hat{\mathcal{S}}=T\mathcal{S}=\mathcal{S}\times (\R^d)^n$ and $\pi(q,v)=q$. We denote $\hat{q}=(q,v)=(x_1,\dots,x_n,v_1,\dots,v_n)$. Again, we want to compute the geodesic equations, first by solving
$$
K_{\pi(\hat{q})}u=\xi_{\pi(\hat{q})}K^V\xi_{\hat{q}}^*p,
$$
then by computing $\partial \partial_{\hat{p}}H(\hat{q},\hat{p},u)=\xi_{\hat{q}}K^V\xi_{\pi(\hat{q})}^*u$ and $\partial_{\hat{q}}H(\hat{q},\hat{p},u)=(\hat{p}-\frac{1}{2}d\pi_{\hat{q}}^*u)
((\partial_{\hat{q}}K_{\hat{q}})d\pi_{\hat{q}}^*u).$ Again, the term $(\hat{p}-d\pi_{\hat{q}}^*u)
(K_{\hat{q}}((\partial_{\hat{q}}d\pi_{\hat{q}}^*)u)$ does not appear in $\partial_{\hat{q}}H(\hat{q},\hat{p},u)$ because $\partial_{\hat{q}}(d\pi_{\hat{q}}^*)=0$.

A momentum on $\hat{\mathcal{S}}$ is denoted $\hat{p}=(p,l)=(p_1,\dots,p_n,l_1,\dots,l_n)$, and we have
$$
\xi_{q,v}^*(p,l)(X)=\sum_{j=1}^np_j(X(x_i))+l_j(\d X_{x_j}(v_j)).
$$
Therefore we get
$$
K^V\xi_{q,v}^*(p,l)(x)=\sum_{j=1}^ne(x-x_j)p_j-\frac{1}{\sigma}\sum_{k=1}^ne(x-x_j)[(x-x_j)\cdot v_j]l_j^T.
$$
Hence, the equation $K_{\pi(\hat{q})}u=\xi_{\pi(\hat{q})}K^V\xi_{\hat{q}}^*p$ becomes the linear problem
$$
\sum_{k=1}^ne(x_j-x_k)u_k^T=\sum_{k=1}^ne(x_j-x_k)p_j^T-\frac{1}{\sigma}\sum_{k=1}^ne(x_j-x_k)[(x_j-x_k)\cdot v_j]l_j^T,\quad j=1,\dots,n,
$$
solved by
$$
u_i=p_i-\frac{1}{\sigma}\sum_{j,k=1}^n S(x_i,x_j)e(x_j-x_k)[(x_j-x_k)\cdot v_k]l_k,\quad i=1,\dots,n.
$$
Then we compute $\xi_{q,v}K^V\xi_{q}^*u=(w_1,\dots,w_n,w'_1,\dots,w'_n)$
$$
\left.\begin{aligned}
w_a&=\ \ \ \ \ \,\sum_{i=1}^ne(x_a-x_i)u_i^T,
\\
&\left(=\ \ \ \ \ \,\sum_{i=1}^ne(x_a-x_i)p_i^T-\frac{1}{\sigma}\sum_{j,k=1}^n e(x_j-x_k)[(x_j-x_k)\cdot v_k]l_k^T\right)\\
w'_a&=-\frac{1}{\sigma}\sum_{i=1}^ne(x_a-x_i)[v_a\cdot (x_a-x_i) ]u_i^T,
\end{aligned}\right\rbrace\qquad a=1,\dots,n.
$$
Now if we denote 
$$
\hat{p}'=\hat{p}-\frac{1}{2}d\pi_{\hat{q}}^*u=(p_1',\dots,p_n',l_1',\dots,l_n')=(p_1-\frac{1}{2}u_1, \dots,p_n-\frac{1}{2}u_n,l_1,\dots,l_n),
$$
we obtain
$$
\hat{p}'(K_{\hat{q}}d\pi^*\hat{u})=\sum_{i,j=1}^ne(x_i-x_j)[p_i'\cdot u_j]-\frac{1}{\sigma}\sum_{i,j=1}^ne(x_i-x_j)[v_i\cdot(x_i-x_j)][l_i\cdot u_j].
$$
Therefore, for $-\partial_{\hat{q}}H^\pi_1(\hat{q},\hat{p},u)=(\alpha_1,\dots,\alpha_n,\beta_1,\dots,\beta_n)$, we get
$$
\left.\begin{aligned}
\alpha_a=&\ \ \ \ \,\frac{1}{\sigma}\,\sum_{j=1}^ne(x_a-x_j)[p'_a\cdot u_j+p'_j\cdot u_a](x_a-x_j)^T
\\
&+\,\frac{1}{\sigma}\,\sum_{j=1}^ne(x_a-x_j)([l_a\cdot u_j]v_a^T-[l_j\cdot u_a]v_j^T)
\\
&-\frac{1}{\sigma^2}\sum_{j=1}^n e(x_a-x_j)[(x_a-x_j)\cdot([l_a\cdot u_j]v_a-[l_j\cdot u_a]v_j)](x_a-x_j)^T,\\
\beta_a=&\ \ \ \ \,\frac{1}{\sigma}\,\sum_{j=1}^ne(x_a-x_j)[l_a\cdot u_j ](x_a-x_j)^T,
\end{aligned}\right\rbrace\qquad a=1,\dots,n.
$$
Since $p'_i=p_i-\frac{1}{2}u_i$, $[p'_a\cdot u_j+p'_j\cdot u_a]=[p_a\cdot u_j+p_j\cdot u_a-u_j\cdot u_a]$ and we obtain the geodesic equations
$$
\left.\begin{aligned}
u_a=&p_a-\frac{1}{\sigma}\sum_{j,k=1}^n S(x_a,x_j)e(x_j-x_k)[(x_j-x_k)\cdot v_k]l_k,
\\
\dot{x}_a=&\ \ \ \ \ \ \,\sum_{i=1}^ne(x_a-x_i)u_i^T,\\
\dot{v}_a=&-\frac{1}{\sigma}\sum_{i=1}^ne(x_a-x_i)[v_a\cdot (x_a-x_i) ]u_i^T
\end{aligned}\right\rbrace\qquad a=1,\dots,n,
$$
and
$$
\left.\begin{aligned}
\dot{p}_a=&\ \ \ \ \,\frac{1}{\sigma}\,\sum_{j=1}^ne(x_a-x_j)[p'_a\cdot u_j+p'_j\cdot u_a](x_a-x_j)^T
\\
&+\,\frac{1}{\sigma}\,\sum_{j=1}^ne(x_a-x_j)([l_a\cdot u_j]v_a^T-[l_j\cdot u_a]v_j^T)
\\
&-\frac{1}{\sigma^2}\sum_{j=1}^n e(x_a-x_j)[(x_a-x_j)\cdot([l_a\cdot u_j]v_a-[l_j\cdot u_a]v_j)](x_a-x_j)^T,\\
\dot{l}_a=&\ \ \ \ \,\frac{1}{\sigma}\,\sum_{j=1}^ne(x_a-x_j)[l_a\cdot u_j ](x_a-x_j)^T,
\end{aligned}\right\rbrace\qquad a=1,\dots,n.
$$

\section*{Conclusion}

This paper gave a comprehensive study of the geometry of abstract shape spaces in the LDDMM framework, and introduced several possible areas of applications, from fibered shapes to lifted shape spaces, to even more general shapes. We also showed that the LDDMM method is, at least on infinite dimensional shape space, largely a sub-Riemannian problem, not a Riemannian one.

Other than those applications, several things still need to be investigated. First of all, the case of weak Riemannian (and even sub-Riemannian) structures, as in \cite{BHM,MM}. Most of the results still work in this case, but they  require additional attention, because the Hamiltonian equations do not always have a solution, so that the geodesic flow must be checked to exist on a case by case basis. The case of Fr\'echet and convenient shape spaces was also only briefly mentionned, and may deserve additional attention.

Another kind of problem left untouched is that of image matching. An image on $M$ is a function $I:M\rightarrow \R$. $\D^s(M)$ acts by composition of the inverse on the right $\varphi\cdot I=I\circ\varphi^{-1}$. This action does not satisfy the shape space hypothesis, and a different point of view must be adopted. Also see \cite{BGR1,BGR2} for another applications of sub-Riemannian geometry to image analysis.

Finally, infinite dimensional sub-Riemannian manifolds are still very poorly understood, particularly because of the appearance of elusive geodesics.

\end{document}